\documentclass[12pt]{article}

\usepackage{amsmath,amsfonts,amssymb,amsthm}
\usepackage{fullpage} % set margins to be reasonable
\usepackage[colorlinks]{hyperref} % add bookmarks to PDF (and some hyperlinks)
\usepackage[parfill]{parskip} % vertical space between paragraphs rather than tabs
\usepackage{enumitem} % specify how enumerates are numbered 
\usepackage{graphicx}
\usepackage{xcolor}
\usepackage{mathtools}
\usepackage{dsfont}

\newtheorem{theorem}{Theorem}[section]
\newtheorem{thm}[theorem]{Theorem}
\newtheorem{lemma}[theorem]{Lemma}

\newtheorem{proposition}[theorem]{Proposition}
\newtheorem{conjecture}[theorem]{Conjecture}

\newtheorem{observation}[theorem]{Observation}
\newtheorem{definition}[theorem]{Definition}

\theoremstyle{remark}

\newcommand{\E}{\mathbb{E}}

\newcommand{\ex}{\text{ex}}

\DeclarePairedDelimiter\floor{\lfloor}{\rfloor}
\DeclarePairedDelimiter\ceil{\lceil}{\rceil}

\title{Lower Bounds on the Erd\H{o}s-Gy\'arf\'as Problem via Color Energy Graphs}
\author{
  J\'ozsef Balogh \footnote{Department of Mathematics, University of Illinois  Urbana-Champaign, Urbana, Illinois 61801, USA. E-mail: \texttt{jobal@illinois.edu}. Research supported by NSF RTG Grant DMS-1937241, NSF Grant DMS-1764123, Arnold O. Beckman Research Award (UIUC Campus Research Board RB 18132), the Langan Scholar Fund (UIUC), and the Simons Fellowship.}
 \and Sean English \footnote {Department of Mathematics, University of Illinois  Urbana-Champaign, Urbana, Illinois 61801, USA. E-mail: \texttt{senglish@illinois.edu}.}
 \and Emily Heath \footnote {Department of Mathematics, Iowa State University, Ames, Iowa 50011, USA. Email: \texttt{eheath@iastate.edu}. Research partially supported by NSF RTG Grant DMS-1937241 and NSF RTG Grant DMS-1839918.}
 \and Robert A.~Krueger \footnote {Department of Mathematics, University of Illinois  Urbana-Champaign, Urbana, Illinois 61801, USA. Email: \texttt{rak5@illinois.edu}. Research partially supported by NSF RTG Grant DMS-1937241, UIUC Campus Research Board Grant RB 18132, and NSF Graduate Research Fellowship Program under Grant No.~DGE 21-46756.}
  }
  
  \date{}

\begin{document}
\maketitle

\begin{abstract}
Given positive integers $p$ and $q$, a $(p,q)$-coloring of the complete graph $K_n$ is an edge-coloring in which every $p$-clique receives at least $q$ colors. Erd\H{o}s and Shelah posed the question of determining $f(n,p,q)$, the minimum number of colors needed for a $(p,q)$-coloring of $K_n$. 
In this paper, we expand on the color energy technique introduced by Pohoata and Sheffer to prove new lower bounds on this function, making explicit the connection between bounds on extremal numbers and $f(n,p,q)$. 
Using results on the extremal numbers of subdivided complete graphs, theta graphs, and subdivided complete bipartite graphs, we generalize results of Fish, Pohoata, and Sheffer, giving the first nontrivial lower bounds on $f(n,p,q)$ for some pairs  $(p,q)$ and improving previous lower bounds for other pairs. 


\end{abstract}

\section{Introduction}\label{sec::intro}

The \emph{Ramsey number} $r_k(p)$ is the minimum number of vertices $n$ for which every edge-coloring of the complete graph $K_n$ with $k$ colors must contain a monochromatic copy of the clique $K_p$. In 1975, Erd\H{o}s and Shelah \cite{E1,E2} introduced the following natural extension of the Ramsey number. Given positive integers $p$ and $q$ with $1\leq q\leq \binom{p}{2}$, a \emph{$(p,q)$-coloring} of $K_n$ is an edge-coloring of $K_n$ in which every $p$ vertices span a clique with at least $q$ colors. The \emph{$(p,q)$-coloring number} of $K_n$, denoted by $f(n,p,q)$, is the minimum number of colors needed to give a $(p,q)$-coloring of $K_n$. Here and throughout, we analyze the asymptotics of $f(n,p,q)$ in $n$, considering $p$, $q$, and related parameters to be constant. The asymptotic notation suppresses dependencies on these constants. It is worth noting that determining the values of $f(n,p,2)$ for all $n$ and $p$ is equivalent to determining the values of the multicolor Ramsey function $r_k(p)$ for all $k$ and $p$.  

These generalized Ramsey numbers were first studied systematically by Erd\H{o}s and Gy\'arf\'as~\cite{EG}. In addition to studying several small cases of $p$ and $q$, they also identified the values of $q$ as a function of $p$ for which $f(n, p, q)$ becomes linear in $n$, quadratic in $n$, and asymptotically equivalent to $\binom{n}{2}$. Namely, they showed that $q=\binom{p}{2}-p+3$ is the smallest value of $q$ for which $f(n,p,q)=\Omega(n)$ and $q=\binom{p}{2}-\lfloor\frac{p}{2}\rfloor+2$ is the smallest value of $q$ for which $f(n,p,q)=\Omega(n^2)$. They also applied the Lov\'asz Local Lemma~\cite{EL} to obtain what is still the best known upper bound on $f(n,p,q)$ for general $p$ and $q$,
\begin{equation}\label{eq::lll}
f(n,p,q)=O\left(n^{\frac{p-2}{\binom{p}{2}-q+1}}\right).
\end{equation}
Additionally, a simple inductive argument was given in~\cite{EG} to prove a lower bound in the diagonal case, i.e. when $p=q$,
\[
f(n,p,p)\geq n^{\frac{1}{p-2}}-1.
\]

Since then, significant progress has been made towards understanding the behavior of this $(p,q)$-coloring function. 
S\'ark\"ozy and Selkow \cite{SS1,SS2} explored the behavior of $f(n,p,q)$ for values of $q$ between the thresholds given in~\cite{EG}. In particular, they proved that there are at most $\log p$ values of $q$ for which $f(n,p,q)$ is linear, and showed that $f(n,p,q)=\binom{n}{2}-o(n^2)$ for all $q>\binom{p}{2}-\lfloor\frac{p}{2}\rfloor+2+\lceil\frac{\log_2p}{2}\rceil$. 

In addition, Conlon, Fox, Lee, and Sudakov \cite{CFLS} generalized constructions of Mubayi and Eichorn~\cite{EM,M1} to prove that $f(n,p,p-1)=n^{o(1)}$ for all $p$, settling the question posed in~\cite{EG} of whether $q = p$ is the smallest value for which $f(n, p, q)$ is polynomial in $n$. The probabilistic upper bound on $f(n,p,p)$ has also been improved for several small values of $p$: Mubayi~\cite{M2} gave a construction for $p=4$, and  Cameron and Heath~\cite{CH,CH2} generalized this approach to improve the bounds for $p\in\{5,6,8\}$.

However, there are many values of $(p,q)$ for which very little is known about $f(n,p,q)$.  Recently, Pohoata and Sheffer~\cite{PS} introduced the color energy of a graph, which they used to obtain lower bounds on the generalized Ramsey numbers for a new family of values of $(p,q)$. 

Given a graph $G=(V,E)$, a set of colors $C$, and an edge-coloring $\chi:E\rightarrow C$, the \emph{color energy} of $G$ is 
\[
\E(G)=|\{(v_1,v_2,v_3,v_4)\in V^4:\chi(v_1,v_2)=\chi(v_3,v_4)\}|.
\]
By bounding the color energy of $K_n$ under a $(k,\binom{k}{2}-m\cdot\left\lfloor\frac{k}{m+1}\right\rfloor+m+1)$-coloring, Pohoata and Sheffer~\cite{PS} proved the following new family of bounds:

\begin{thm}[\cite{PS}]\label{thm::PS}
For any integers $k>m\geq 2$, 
\[
f\left(n,k,\binom{k}{2}-m\cdot\left\lfloor\frac{k}{m+1}\right\rfloor+m+1\right)=\Omega\left(n^{1+\frac{1}{m}}\right).
\]
\end{thm}

In \cite{FPS}, Fish, Pohoata, and Sheffer further developed the color energy approach, defining higher color energies and color energy graphs as tools to prove additional families of bounds. Color energy graphs are the main focus of this paper, and the bounds on $f(n,p,q)$ which are proven using color energy graphs all fit into a general framework, which we describe here.

Let $F$ be a bipartite graph and $r\geq 2$, $\alpha>0$ be constants. We say $F$ is \emph{$(r,\alpha)$-nice} if  $\mathrm{ex}(n,F)=O(n^{2-\alpha})$ and
\begin{equation}\label{equation r alpha nice}
f\left(n,r|V(F)|,\binom{r|V(F)|}2-(r-1)|E(F)|+1\right)=\Omega\left(n^{\frac{\alpha r}{r-1}}\right).
\end{equation}

Furthermore, we say that $F$ is simply \emph{$r$-nice} if $F$ is $(r,\alpha)$-nice for every $\alpha>0$ such that $\mathrm{ex}(n,F)=O(n^{2-\alpha})$. Intuitively, a graph $F$ is $r$-nice if any coloring of $K_n$ with significantly fewer than $n^{\frac{\alpha r}{r-1}}$ colors has a clique of size $r|V(F)|$ spanning as few colors as we would expect if there were $r$ pairwise vertex-disjoint copies of $F$ each with the same coloring.

In \cite{FPS}, the authors used the color energy graph to prove that even cycles $C_{2k}$ are $2$-nice, which led them to find new bounds on $f(n,p,q)$ based on known upper bounds on $\mathrm{ex}(n,C_{2k})$. They also showed that $C_8$ is $3$-nice, but did not prove that any infinite families of graphs are $r$-nice for any $r\geq 3$. We extend the techniques used in \cite{FPS} to find many families of $(r,\alpha)$-nice graphs. For $t\geq 3$, denote by $K_t^+$  the subdivision of $K_t$ and by $\varTheta(a,b)$ the \emph{theta graph}, which consists of $b$ internally-disjoint paths of $a$ edges each with the same two endpoints. Furthermore, for $a\leq b$, let $K_{a,b}^{\ell}$ denote the graph obtained from $K_{a,b}$ by replacing each edge with a path of length $\ell$ (that is, $K_{a,b}^1=K_{a,b}$).
\begin{theorem}\label{thm::nice}
We have the following:
\begin{enumerate}
    \item[(i)] \makebox[12.8cm]{$K_t^+$ is $2$-nice for all $t\geq 3$.\hfill}(Theorem~\ref{thm::subKt}) 
    \item[(ii)] \makebox[12.8cm]{$\varTheta(a,b)$ is $(r,1-\frac{1}{a})$-nice for all $r,b\geq 2$, $a> r,$ and $r$-nice when $b=2$\hfill}
    \vspace{-.15cm}\item[] \makebox[12.8cm]{or $b$ is sufficiently large compared to $a$.\hfill}(Theorem~\ref{thm::theta})
    \item[(iii)] \makebox[12.8cm]{$K_{3,b}^\ell$ is $(r,1-\frac{2}{3\ell})$-nice for $b\geq 3$, $\ell\geq 2$,  $r\in \{2,3,4,5,6\}$ with $r<\frac{3}2\ell$.\hfill}(Theorem~\ref{thm::subKtt})
\end{enumerate}
\end{theorem}

Each of these families gives us new lower bounds on $f(n,p,q)$ for the appropriate choices of $p$ and $q$, many of which improve on existing bounds in the literature. For a detailed comparison of our bounds to previous results, see Section~\ref{sec::intro:comparison}.

The fact that $\varTheta(a,2)=C_{2a}$ is $r$-nice for $a>r$ is particularly interesting due to the connection with the long-standing question about lower bounds for the extremal number $\ex(n,C_{2a})$. This extremal number is only known up to a multiplicative constant factor when $a\in \{2,3,5\}$. If one could show that
\begin{equation}\label{equation cycle contrapositive bound}
f\left(n,r|V(C_{2a})|,\binom{r|V(C_{2a})|}2-(r-1)|E(C_{2a})|+1\right)=O\left(n^{\frac{\alpha r}{r-1}}\right),
\end{equation}
for some $\alpha>0$ and $a>r$, then Theorem~\ref{thm::theta} would imply that $\ex(n,C_{2a})=\Omega(n^{2-\alpha})$. In particular, proving \eqref{equation cycle contrapositive bound} for $\alpha=1-\frac{1}a$ would give a lower bound on $\ex(n,C_{2a})$ matching the upper bound up to a multiplicative constant factor. It is unclear, however, if determining the value of $f(n,p,q)$ for this specific choice of parameters would be easier than proving the extremal number for even cycles directly.

It is worth noting that Theorem~\ref{thm::PS} applied with $m=r-1$, and $k=rt$ for some $t$ implies that trees $T$ with $|V(T)|=t$ are $r$-nice, as it is well-known that $\mathrm{ex}(n,T)=\Theta(n)$.  

Using a simple induction argument employed by Erd\H{o}s and Gy\'arf\'as~\cite{EG} in their original paper, we improve a result of Fish, Pohoata, and Sheffer~\cite{FPS} and obtain a generalization of the lower bound from~\cite{EG} on $f(n,p,p)$.

\begin{thm}\label{thm::induction}
For every $1 \leq m \leq k-1$, we have
\[ 
f\left( n, k, \binom{k}{2} - m(k-m) - \binom{m}{2} + m + 1 \right) = \Omega\left(n^{1/m}\right).
\]
\end{thm}

Moving forward, when considering $f(n,p,q)$, instead of viewing the problem as requiring at least $q$ colors on every $p$-clique, it is often helpful to think about the problem in terms of having at most $\binom{p}2-q$ repetitions among existing colors. More formally, if $K_p$ is an edge-colored clique and $C$ is the set of all colors that appear on edges of $K_p$, we say that the clique has $\binom{p}2-|C|$ \emph{color repetitions} or just \emph{repetitions}. For accounting purposes, it will often be useful to count repetitions in groups. For example, if we consider a set of $k$ edges of the same color $c$ in a $p$-clique, we will count this as $k-1$ repetitions, unless we have already discovered a different edge of color $c$, in which case we will consider these $k$ edges to have yielded $k$ repetitions.

\subsection{Relation to the Conlon-Tyomkyn Problem}\label{sec::intro:conlontyomkyn}

Motivated by the treatment of the Erd\H{o}s-Gy\'arf\'as problem in~\cite{K}, Conlon and Tyomkyn~\cite{CT} asked how many colors are necessary in a proper edge-coloring of $K_n$ without many vertex-disjoint color-isomorphic copies of some fixed small graph $H$. (We say edge-colored graphs are color-isomorphic if there is an isomorphism between them preserving the colors.) More precisely, for $n, k \geq 2$ and a graph $H$, we denote by $f_k(n, H)$ the smallest integer $C$ such that there is a proper edge-coloring of $K_n$ with $C$ colors containing no $k$ disjoint color-isomorphic copies of $H$. They considered only proper colorings to avoid some trivial obstacles, but the connection to determining $f(n,p,q)$ can be made explicit. 

Fix $n,k\geq 2$ and a graph $H$ such that  $(k-1)|E(H)|\leq k|V(H)|-2$. Let $p=k|V(H)|$ and $q=\binom{k|V(H)|}{2}-(k-1)|E(H)|+1$, and consider a $(p,q)$-coloring of $K_n$. While this coloring is not necessarily proper, it forbids monochromatic stars on $p$ vertices, since such a star would contain $k|V(H)|-2$ repetitions, which is more than allowed by the $(p,q)$-coloring.  Since each vertex is incident to a bounded number of edges of each color, we can obtain a proper coloring by expanding our set of colors by a constant factor. This new coloring cannot contain $k$ disjoint color-isomorphic copies of $H$, otherwise we can find a $p$-clique in the original coloring with fewer than $q$ colors. Therefore, our $(p,q)$-coloring must use $\Omega(f_k(n,H))$ colors, and hence, we have
\begin{equation}\label{eq::CT}
f\left(n,k|V(H)|,\binom{k|V(H)|}{2} - (k-1)|E(H)| + 1\right) = \Omega(f_k(n,H)).
\end{equation}

By exploiting this relationship between the two problems, we can obtain bounds on the Erd\H{o}s-Gy\'arf\'as function $f(n,p,q)$ using known results about $f_k(n,H)$. For example, Conlon and Tyomkyn~\cite{CT} gave a short proof that for every integer $k$ and tree $T$ with $m$ edges, $f_k(n,T)=\Omega(n^{1+1/m})$, which allows us to recover Theorem~\ref{thm::PS} without the need to invoke color energy. More recently, Xu and Ge~\cite{GX} showed that for $t\geq3$, $f_2(n, K_t^+) = \Omega(n^{1+\frac{1}{2t-3}})$. Applying \eqref{eq::CT} gives a result which matches Theorem~\ref{thm::subKt} due to our current knowledge of the extremal number of $K_t^+$. 
Among other impressive results, Janzer~\cite{J1} showed that for fixed integers $k,r\geq 2$, 
\[
f_r(n,C_{2k})=\Omega\left(n^{\frac{r}{r-1}\cdot\frac{k-1}{k}}\right).
\]
His proof can be extended in a straightforward manner to show
\[
f_r(n,\varTheta(a,b))=\Omega\left(n^{\frac{r}{r-1}\cdot\frac{a-1}{a}}\right),
\]
which matches our result in Theorem~\ref{thm::theta}.

\subsection{Comparison with Previous Results}\label{sec::intro:comparison}

In addition to giving non-trivial lower bounds on $f(n,p,q)$ for new families of pairs $(p,q)$, our work improves existing bounds for previously studied families of pairs $(p,q)$. To see how our theorems improve existing bounds, note that $f(n,p,q)\leq f(n,p,q')$ for $q\leq q'$. By fixing a number of vertices $p$ and a total number of colors, we can compare results by considering the number of repetitions that we can guarantee on each $p$-clique. 

Setting $m=2t-3$ and $k=2s$ in Theorem~\ref{thm::PS} \cite{PS} gives
\[
f\left(n,2s,\binom{2s}{2}-\left(t^2-\frac{5}{2}t+\frac{3}{2}\right)+1\right)=\Omega
\left(n^{1+\frac{1}{2t-3}}\right), 
\]
which we improve in Theorem~\ref{thm::subKt}, showing 
\[
f\left(n,2s,\binom{2s}{2}-(t^2-t)+1\right)=\Omega
\left(n^{1+\frac{1}{2t-3}}\right). 
\]
 
We can perform a similar comparison between Theorems~\ref{thm::subKtt} and~\ref{thm::PS} by considering the case when the number of colors $n^{1+1/m}$ satisfies  
\[
m=\frac{3(r-1)\ell}{3\ell-2r}
\]
for some $r<\frac{3}{2}\ell$. 
In this case, when $r\leq \ell$, our theorem gives the same number of repetitions on $k=r(3+(3\ell-2)b)$ vertices as Theorem~\ref{thm::PS}. However, when $r>\ell$, Theorem~\ref{thm::PS} gives 
\[
f\left(n,k,\binom{k}{2}-3(r-1)\ell b+m+1\right)=\Omega\left(n^{1+1/m}\right)
\] while our Theorem~\ref{thm::subKtt} improves this to
\[
f\left(n,k,\binom{k}{2}-3(r-1)\ell b +1\right)=\Omega\left(n^{1+1/m}\right).
\] 

Similarly, Theorem~\ref{thm::induction} improves the following result of Fish, Pohoata, and Sheffer~\cite{FPS}, proved using an application of the well-known K\H{o}v\'ari-S\'os-Tur\'an Theorem~\cite{KST}.  

\begin{thm}[\cite{FPS}]\label{thm::FPS no color energy}
For any integers $2\leq m\leq k/2$, 
    \[
    f\left(n, k, \binom{k}{2}-m(k-m)+2\right)=\Omega\left(n^{1/m}\right).
    \]
\end{thm}

For comparison, in Theorem~\ref{thm::induction}, we obtain the same bound for a smaller number of colors on each clique:  
\[
f\left( n, k, \binom{k}{2} - m(k-m) - \binom{m}{2} + m + 1 \right) = \Omega\left(n^{1/m}\right).
\]

\begin{proof}[Proof of Theorem~\ref{thm::induction}]
Suppose we color $K_n$ with $c:=n^{1/m}/(2(k-m)^{1/m})=\Theta(n^{1/m})$ colors. Arbitrarily choose a vertex $v_1$, and note that there exists a color such that $v_1$ is incident with at least 
\[
\frac{n-1}{n^{1/m}/(2(k-m)^{1/m})}>(k-m)^{1/m}n^{1-1/m}
\]
edges of this color, say color 1. Restricting to the neighborhood of $v_1$ in color 1, arbitrarily choose a vertex $v_2$, and we can find a color such that $v_2$ is incident with at least 
\[
\frac{(k-m)^{1/m}n^{1-1/m}-1}{n^{1/m}/(2(k-m)^{1/m})}>(k-m)^{2/m}n^{1-2/m}
\]
edges of this color, say color 2, all of whose endpoints are in the color 1 neighborhood of $v_1$. Continue iteratively, until we have selected vertices $v_1, v_2,\dots, v_m$ and colors $1,2,\dots,m$ such that there at least
\[
k-m=(k-m)^{m/m}n^{1-m/m}
\]
vertices simultaneously in the $i$-th color-neighborhood of $v_i$ for all $1\leq i\leq m$. Then we have a set of $k$ vertices that spans at most \[
m+\binom{k-m}2=\binom{k}2-m(k-m)-\binom{m}2+m
\]
colors. Thus, any coloring which hopes to have every clique span one more color needs at least $\Omega(c)=\Omega(n^{1/m})$ colors, completing the proof.
\end{proof}

Many of our results are essentially incomparable with previous results, but their quality can be judged with the local lemma bound in~\eqref{eq::lll}. This and Theorem~\ref{thm::theta} state that for every $r,b\geq 2$ and $a>r$, for $\ell = 2 + b(a-1)$, there exists $C,c>0$ such that for $n$ sufficiently large,
\[
c n^{\frac{r}{r-1} \cdot \frac{a-1}{a}} \leq f\left(n, r\ell, \binom{r\ell}{2} - (r-1)ab + 1\right) \leq C n^{\frac{r}{r-1} \cdot \frac{a-1}{a} + \frac{2}{ab}} .
\]
Note that as $b$ increases, the gap between the lower and upper bounds shrinks (although $C$ and $c$ implicitly depend on $b$).

\subsection{Organization}

The rest of the paper is organized as follows. The proofs of Theorems~\ref{thm::subKt}, \ref{thm::theta}, and \ref{thm::subKtt} increase in difficulty, so as we develop the concept of the color energy graph and associated tools, we prove these theorems when we have sufficient techniques to do so. In Section~\ref{sec::color energy}, we define the color energy graph and a helpful variant called the pruned color energy graph (Section~\ref{sec::color energy:pruned}). These tools are sufficient to provide a short proof of Theorem~\ref{thm::subKt} in Section~\ref{sec::color energy:subKt}. In Section~\ref{sec::homo}, we develop terminology and theory necessary for finding more complicated structures with the color energy graph. Following this development, we prove Theorem~\ref{thm::theta} in Section~\ref{sec::theta}. Finally, in Section~\ref{sec::subKtt}, we give the more involved proof of Theorem~\ref{thm::subKtt}. We provide some avenues for future research in Section~\ref{sec::conc}.

\section{Color Energy Graph}\label{sec::color energy}

In analogy to the additive energy of additive combinatorics, Pohoata and Sheffer \cite{PS} defined color energy of an edge-colored graph. With Fish \cite{FPS}, they went further in defining a corresponding graph, the ``color energy graph.'' As we use this graph extensively, we collect here its definition, some basic results, and a few helpful modifications to the color energy graph.

\begin{definition}\label{def::color energy}
Given a graph $G = (V,E)$ with a coloring $\chi: E \to C$, the \emph{$r$-th color energy graph} $\vec{G} = (\vec{V}, \vec{E})$ has vertex set $V^r$ with an edge between $(v_1, \dots, v_r)$ and $(u_1, \dots, u_r)$ if and only if $\chi(v_1u_1) = \ldots = \chi(v_ru_r)$.
\end{definition}

If it is clear from context, we will omit the $r$-th in the name and simply refer to $\vec{G}$ as the color energy graph. 
Note that $\vec{V}$ includes $r$-tuples with repeated coordinates, but $\vec{G}$ is loopless as $G$ is loopless.\footnote{Fish, Pohoata, and Sheffer~\cite{FPS} defined the $r$-th color energy graph as a subgraph of $\vec{G}$ obtained by removing certain edges such as loops and edges of ``unpopular'' colors. For clarity, we instead gather the necessary restrictions on $\vec{E}$ for our proofs in the definition of a pruned energy graph in Section~\ref{sec::color energy:pruned}.} 
Since an edge of the color energy graph corresponds to a multiset of $r$ edges of the same color in $G$, the coloring $\chi$ also naturally extends to a coloring on $\vec{G}$. The following relation between $|E|$, $|\vec{E}|$, and $|C|$ allowed the authors in~\cite{FPS} to derive lower bounds on $f(n,p,q)$.

\begin{proposition}[\cite{FPS}]\label{prop::color energy}
If $G = (V,E)$ with coloring $\chi: E \to C$ has $r$-th color energy graph $\vec{G} = (\vec{V}, \vec{E})$, then
\[ 
|C| \geq \left(\frac{|E|^r}{|\vec{E}|}\right)^{\frac{1}{r-1}}.
\]
\end{proposition}

\begin{proof}
For each color $c \in C$, let $m_c$ be the number of edges of color $c$ in $G$. Observe that $\sum_{c\in C} m_c = |E|$ and $\sum_{c\in C} m_c^r = |\vec{E}|$. H\"older's inequality implies
\[ |\vec{E}| = \sum_{c \in C} m_c^r \geq \frac{\left(\sum_{c\in C} m_c\right)^r}{\left(\sum_{c\in C} 1\right)^{r-1}} = \frac{|E|^r}{|C|^{r-1}} .\]
\end{proof}

We fix some terminology here. Let $G$ and $H$ be graphs. A \emph{graph homomorphism} from $H$ to $G$ is a function $\phi: V(H) \to V(G)$ such that if $uv \in E(H)$ then $\phi(u)\phi(v) \in E(G)$. Such a function on $V(H)$ induces a function on $E(H)$ which we also call $\phi$. As a graph, $\phi(H)$ is called a \emph{homomorphic image} (or just \emph{image}) of $H$ in $G$. If $\phi$ is injective, then $\phi(H)$ is called an \emph{isomorphic copy} (or just \emph{copy}) of $H$ in $G$. If the edges of $H$ are colored and $\phi(H)$ inherits this coloring, then $\phi(H)$ is called a \emph{color-homomorphic image} (likewise \emph{color-isomorphic copy}) of $H$ in $G$.

All of the proofs using the color energy graph follow roughly the same format. We fix a graph $H$, start with an $(r|V(H)|,\binom{r|V(H)|}{2} - (r-1)|E(H)|+1)$-coloring $\chi$ of a complete graph $K_n = (V,E)$, and consider the $r$-th color energy graph $\vec{G} = (\vec{V}, \vec{E})$. Let $\pi_k: \vec{V} \to V$ be the \emph{$k$-th coordinate map} for $1 \leq k \leq r$. If a copy $\vec{H}$ of $H$ is found as a subgraph of $\vec{G}$, then $\pi_k(\vec{H})$ is a color-homomorphic image of $\vec{H}$ for each $k \in [r]$. If these $r$ color-homomorphic images of $\vec{H}$ are actually pairwise vertex-disjoint color-isomorphic copies of $\vec{H}$, then taken together, these copies span $r|V(H)|$ vertices and contain at least $(r-1)|E(H)|$ repetitions, contradicting that $\chi$ is an $(r|V(H)|,\binom{r|V(H)|}{2} - (r-1)|E(H)|+1)$-coloring. In this case, $\vec{H}$ is not a subgraph of $\vec{G}$, so we have an upper bound on $|\vec{E}|$ in terms of $|\vec{V}|$ via the extremal number of $H$. By Proposition~\ref{prop::color energy}, this gives a lower bound on the number of colors used by $\chi$, and hence a lower bound on $f(n,p,q)$.

Unfortunately, it is not guaranteed that the images $\pi_k(\vec{H})$ are disjoint color-isomorphic copies of $\vec{H}$. This lead the authors in~\cite{FPS} to ``prune'' the color energy graph so that, for particular $H$, the color-homomorphic images still contain a sufficient number of repetitions and vertices. We concisely describe this pruning in the next subsection, as we use this for our proofs utilizing the color energy.

\subsection{Pruned Color Energy Graph}\label{sec::color energy:pruned}

Let $K$ be a graph. We say that a homomorphism $\phi: V(H) \to V(G)$ (and the homomorphic image $\phi(H)$) is \emph{$K$-preserving} if every copy of $K$ in $H$ is mapped to a copy of $K$ in $G$ under $\phi$. We denote by $P_k$ the path on $k$ vertices.

We cannot guarantee that the existence of a graph $\vec{H} \subseteq \vec{G}$ will yield $k$ disjoint color-isomorphic copies of $\vec{H}$, but we can guarantee that these color-homomorphic images of $\vec{H}$ are disjoint, bipartite, and $P_3$-preserving.

\begin{definition}\label{def::pruned energy}
Let $G = (V,E)$ be a graph with coloring $\chi : E \to C$. A \emph{pruned $r$-th energy graph} is a subgraph $\vec{G}' = (\vec{V}', \vec{E}')$ of the $r$-th color energy graph $\vec{G} = (\vec{V}, \vec{E})$ with the following structure:
\begin{enumerate}
\item There is a partition $V = V_1 \cup \cdots \cup V_r$ such that $\vec{V}' = V_1 \times \cdots \times V_r$ and $\floor{|V|/r} \leq |V_i| \leq \ceil{|V|/r}$ for each $i$.
\item For every $i$, there exist partitions $V_i = V_i' \cup V_i''$ such that the $i$-th coordinate of every edge of $\vec{E}'$ has one endpoint in each of $V_i'$ and $V_i''$.
\item If $\vec{x}, \vec{y} \in \vec{V}'$ are at distance at most 2 in $\vec{G}'$, then $\vec{x}$ and $\vec{y}$ are not equal in any coordinate.
\end{enumerate}
\end{definition}

We now fix some notation. In general, the vertices of a color energy graph are denoted with a vector arrow above them, as in $\vec{v}$, to remind the reader that they are tuples of vertices of $G$. We speak of entries in the tuples as `coordinates.' We let $\pi_k : V(\vec{G}) \to V_k$ be the $k$-th coordinate map, which induces a map on the edges of $\vec{G}$ as well. We abuse notation and also call this edge map $\pi_k$. Note that by property 3, edges in the pruned energy graph $\vec{G}'$ are sent to edges in $G$ under $\pi_k$. Furthermore, we let $\pi:\vec{G}' \to G$ be defined by $\pi(\vec{H}) = \bigcup_k \pi_k(\vec{H})$. We often consider paths in $\vec{G}'$. By property 3, the image of a path under $\pi_k$ is a walk which may repeat edges and vertices, but will never repeat the same edge consecutively, that is, it will never `turn around.'

The utility of the pruned energy graph is that the additional conditions come at no cost to the growth rate of the number of edges, as long as we consider colorings in which every vertex is incident to a bounded number of edges of each color. To enforce this color degree condition, we will require all our colorings to be $(p,\binom{p}{2}-p+3)$-colorings, which forbids monochromatic stars on $p$ vertices. Unfortunately, this implies that we can only use a pruned color energy graph when we want to prove bounds involving a superlinear number of colors since $(p,\binom{p}{2}-p+3)$-colorings of $K_n$ require $\Omega(n)$ colors. The existence of a pruned energy graph was shown in \cite{FPS}, so we only give a sketch of the proof below.

\begin{proposition}[\cite{FPS}]\label{prop::pruned exists}
Let $G = (V,E)$ be a graph with a $(p,\binom{p}{2}-p+3)$-coloring $\chi: E \to C$, and let $\vec{G} = (\vec{V}, \vec{E})$ be the $r$-th color energy graph. Then, there exists a pruned $r$-th energy graph $\vec{G}' = (\vec{V}', \vec{E}')$ such that $|\vec{V}'| = \Theta(|\vec{V}|)$ and $|\vec{E}'| = \Theta(|\vec{E}|)$.
\end{proposition}

\begin{proof}[Proof Sketch]
A standard probabilistic argument shows that every graph has a bipartite subgraph with at least half as many edges as the original. To get the partition $V_1, \dots, V_r$, one can use a modification of this argument to ensure that a constant fraction of the edges of $\vec{G}$ have all coordinates with both endpoints in the same $V_i$. One achieves property 2 through a similar probabilistic argument.

For property 3, we need that every vertex is incident to at most $p-2$ edges of each color (which is true since $\chi$ is a $(p,\binom{p}{2}-p+3)$-coloring). Among the neighbors of a vertex $\vec{z}$ with first coordinate $v$, there are at most $p-2$ distinct second coordinates, otherwise $\pi_2(\vec{z})$ is incident to at least $p-1$ edges of color $\chi(\pi_1(\vec{z}) v)$. Thus, $\vec{z}$ has at most $(p-2)^{r-1}$ neighbors with first coordinate $v$. We keep one of them for each $\vec{z}$ and $v$, giving property 3 for vertices at distance 2. Since $G$ is loopless, property 3 also holds for vertices at distance 1. 
\end{proof}

\subsection{Subdivided Cliques}\label{sec::color energy:subKt}

In this section, we show that $K_t^+$ is $2$-nice. Using the bound 
\[
\ex(n,K_t^+)=O\left(n^{\frac{3}{2}-\frac{1}{4t-6}}\right)\]
given by Janzer~\cite{J3} for $t\geq 3$, we obtain the following result. 

\begin{theorem}\label{thm::subKt}
Let $t\geq 3$. Then $K_t^+$ is $2$-nice. Consequently, if
$s:=|V(K_t^+)|=t+\binom{t}2$ (note that $|E(K_t^+)|=2\binom{t}2)$, then we have
\[
f\left(n,2s,\binom{2s}{2}-2\binom{t}{2}+1\right)=\Omega\left(n^{1+\frac{1}{2t-3}}\right).
\]
\end{theorem}

\begin{proof} 
Let $\alpha>0$ be such that $\ex(n,K_t^+)=O(n^{2-\alpha})$, and let $G=(V,E)$ be a complete graph $K_n$, $C$ be a set of colors, and $\chi:E\rightarrow C$ be a $(2s,\binom{2s}{2}-2\binom{t}{2}+1)$-coloring of $G$, where $s=|V(K_t^+)|=t+\binom{t}{2}$. Consider the pruned 2nd energy graph $\vec{G}=(\vec{V},\vec{E})$. Assume that there is a copy $\vec{K}$ of $K_t^+$ in $\vec{G}$ with vertices $\vec{x}_1,\ldots,\vec{x}_t$ and $\vec{y}_1,\ldots,\vec{y}_{\binom{t}{2}}$. Recall that $\pi_k(\vec{K})$ is $P_3$-preserving and bipartite for $k=1,2$. Then clearly no $\vec{x}_i$ and $\vec{y}_j$ share any common coordinates, because of the bipartite structure of $\vec{G}$, and all of the coordinates of $\vec{x}_1,\ldots,\vec{x}_t$ are distinct by the $P_3$-preserving property. While it is possible for the same coordinate to appear in multiple vertices $\vec{y}_1,\ldots, \vec{y}_{\binom{t}{2}}$, corresponding to ``degenerate'' homomorphisms of $K_t^+$ in $G$, this will not change the number of distinct edges that we find in the two copies of $K_t^+$ in $G$. (If it did, then it would happen because two vertices $\vec{y}_i$ and $\vec{y}_j$ adjacent to the same vertex $\vec{x}_k$ shared a coordinate, which is forbidden in a pruned energy graph.) So, $\pi_k(\vec{K})$ has $2\binom{t}2$ edges for $k=1,2$, giving us a clique on at most $2s$ vertices with $2\binom{t}2$ color repetitions, contradicting the assumption about our coloring. Therefore, we have
\[ 
|E(\vec{G})|\leq\ex(n^2,K_t^+)=O\left(n^{4-2\alpha}\right),
\]
and hence by Propositions~\ref{prop::color energy} and \ref{prop::pruned exists}, $|C|=\Omega\left(n^{2\alpha}\right)$. Thus,  $K_t^+$ is $2$-nice, and letting $\alpha=1/2+1/(4t-6)$ yields the result.
\end{proof}

In the above proof we only considered the 2nd energy graph. That is because in order for the $r$-th pruned energy graph to exist, by Proposition~\ref{prop::pruned exists}, we need
\begin{equation}\label{equation subKt conditoin}
r|V(K_t^+)| - 3 \geq (r-1)|E(K_t^+)| - 1 ,
\end{equation}
as otherwise we cannot guarantee that the color degrees of each vertex are bounded. Inequality~\eqref{equation subKt conditoin} is only true for all $t$ when $r=2$. For larger $r$, some values of $t$ still satisfy \eqref{equation subKt conditoin}, but these choices of $t$ either do not give new bounds or are covered elsewhere. For example, setting $r=3$ and $t=3$ gives $K_3^+=C_6$, which is covered in Theorem~\ref{thm::theta}.

The proof of Theorem~\ref{thm::subKt} was relatively straightforward because $P_3$-preserving homomorphisms of $K_t^+$ are easy to understand: all the vertices corresponding to the original $K_t$ must be distinct, and the ``subdivision'' vertices only coincide when the edges of the original $K_t$ that correspond to these vertices form a matching. In particular, all $P_3$-preserving homomorphisms of $K_t^+$ have the same number of edges, which was useful in the proof. For other structures, including longer subdivisions of $K_t$, $P_3$-preserving homomorphisms are not so easily understood, and may not have the same number of edges as the original graph. In the next section, we develop tools to help us analyze $P_3$-preserving homomorphic images of general graphs, allowing us to apply the color energy techniques to theta graphs and subdivided complete bipartite graphs (where the edges are subdivided an arbitrary number of times).

\section{Analyzing Color-Homomorphic Images}\label{sec::homo}

In this section, we develop a framework that will help us prove Theorems~\ref{thm::theta} and \ref{thm::subKtt}, and also may be useful for proving further results outside the scope of this paper.

As usual, let $G$ be an edge-colored graph and $\vec{G}$ be a pruned $r$-th energy graph of $G$. Let $H\subseteq G$ and $\vec{T}\subseteq \vec{G}$ with $m:=|E(\vec{T})|$. We denote by $H_k=H\cap G[V_k]$ the $k$-th coordinate of $H$, for each $1\leq k\leq r$. An ordering $\sigma$ of $E(\vec{T})=\{\vec{e}_1,\dots,\vec{e}_{m}\}$ is called \emph{$H$-compatible} if for every $i\in [m]$, there exists an endpoint $\vec{v}$ of $\vec{e}_i$ such that 
\[
\pi(\vec{v})\subseteq H\cup\bigcup_{j=1}^{i-1} \pi(\vec{e}_j).
\]

\subsection{Graph Revealing Algorithm}\label{sec::homo:alg}

Given graphs $H\subseteq G$ and $\vec{T}\subseteq \vec{G}$, we wish to understand the number of vertices and the number of repetitions in $H\cup\pi(\vec{T})$. To do this, we give a simple algorithm that adds vertices from $\pi(\vec{T})$ to $H$ in $|E(\vec{T})|$ steps, given an $H$-compatible ordering of $E(\vec{T})$. During the algorithm we keep track of several parameters, which in the proofs of Theorems~\ref{thm::theta} and \ref{thm::subKtt} allow us to leverage the structure of $\vec{T}$ and $H$ to analyze the number of vertices and repetitions $\pi(\vec{T})$ adds to $H$.

Let $m:=|E(\vec{T})|$, and let $E(\vec{T})=\{\vec{e}_1,\dots,\vec{e}_{m}\}$, where the ordering of the edges, say $\sigma$, is $H$-compatible. Let $\vec{e}_i=\vec{u}_i\vec{v}_i$, where $\pi(\vec{u}_i)\subseteq H\cup\bigcup_{j=1}^{i-1}\pi(\vec{e}_{j})$. We recursively build graphs $H = H^{(0)}, H^{(1)}, \dots, H^{(m)} = H \cup \pi(\vec{T})$, where for $1 \leq i \leq m$,
\[
H^{(i)} = H^{(i-1)}\cup\pi(\vec{e}_i).
\]

We sequentially add the edges of $\vec{T}$ to $H$ according to the $H$-compatible ordering of $E(\vec{T})$. We now define some terminology which will be useful for analyzing this process.
\begin{itemize}
	\item We say that step $i$ gives us a \emph{new vertex in coordinate $k$} if the vertex $\pi_k(\vec{v}_i)\not\in V(H_k^{(i-1)})$ (which implies that the edge $\pi_k(\vec{e}_i)\not\in E(H_k^{(i-1)})$). Let $n_{i,k}(\vec{T},H,\sigma):=1$ if we get a new vertex in coordinate $k$ on step $i$, and $n_{i,k}(\vec{T},H,\sigma):=0$ otherwise.
	\item We say that step $i$ gives us a \emph{savings in coordinate $k$} if the vertex $\pi_k(\vec{v}_i)\in V(H_k^{(i-1)})$, but the edge $\pi_k(\vec{e}_i)\not\in E(H_k^{(i-1)})$. Let $s_{i,k}(\vec{T},H,\sigma):=1$ if we get a savings in coordinate $k$ on step $i$, and $s_{i,k}(\vec{T},H,\sigma):=0$ if we do not.
	\item We say that step $i$ gives us a \emph{delayed vertex in coordinate $k$} if the edge $\pi_k(\vec{e}_i)\in E(H_k^{(i-1)})$ (which implies that the vertex $\pi_k(\vec{v}_i) \in V(H_k^{(i-1)})$). Set $d_{i,k}(\vec{T},H,\sigma):=1$ if we get a delayed vertex in coordinate $k$ on step $i$, and $d_{i,k}(\vec{T},H,\sigma):=0$ otherwise.
\end{itemize}
Often $\vec{T}$, $H$ and $\sigma$ will be clear from context. In those cases, we will omit them from the notation. By definition, $n_{i,k} + s_{i,k} + d_{i,k} = 1$. We now define aggregate parameters derived from $n_{i,k}$, $s_{i,k}$ and $d_{i,k}$. Let
\[
\begin{array}{ccccc}
	n_i:=\displaystyle\sum_{k=1}^r n_{i,k}, & & N:=\displaystyle\sum_{i=1}^{m}n_i, & & N_k:=\displaystyle\sum_{i=1}^{m} n_{i,k}, \\
	s_i:=\displaystyle\sum_{k=1}^r s_{i,k}, & & S:=\displaystyle\sum_{i=1}^{m}s_i, & & S_k:=\displaystyle\sum_{i=1}^{m}s_{i,k}, \\
	d_i:=\displaystyle\sum_{k=1}^r d_{i,k}, & & D:=\displaystyle\sum_{i=1}^{m}d_i, & &
	D_k:=\displaystyle\sum_{i=1}^{m}d_{i,k}\\
\end{array}
\]
be the new vertices, savings, and delayed vertices in step $i$, in total, and in each coordinate, respectively. We emphasize here that each of these parameters are functions of $\vec{T}$, $H$ and $\sigma$, so we may write $N(\vec{T},H,\sigma)$, $S_k(\vec{T},H,\sigma)$ or other parameters with these variables in cases where the triple $(\vec{T},H,\sigma)$ is not clear from context. Summing $n_{i,k} + s_{i,k} + d_{i,k} = 1$ over $k$, we get for all $1\leq i\leq m$
\begin{equation}\label{eq::add to r}
	n_i+s_i+d_i=r,
\end{equation}
and summing this over $i$, we get
\[ N + S + D = rm .\]

Finally, let
\[ d:=\sum_{i=1}^{m} \mathds{1}_{(d_i=0)} \]
be the number of steps where there are no delayed vertices.

We can precisely describe the number of vertices and repetitions added to $H$ by $\vec{T}$ under the ordering $\sigma$ using these parameters. By the definition of $N$, we have
\begin{equation}\label{eq::vertices from revealing}
	|V(H\cup\pi(\vec{T}))|=|V(H)|+N=|V(H)|+rm-S-D.
\end{equation}
Furthermore, if $R$ is the number of repetitions in $H$ and $R^*$ is the number of repetitions in $H\cup \pi(\vec{T})$, then
\begin{equation}\label{eq::repeats from revealing}
	R^*-R\geq \sum_{i=1}^{m} \left(n_i+s_i-\mathds{1}_{(d_i=0)}\right)=rm-D-d.
\end{equation}
Indeed, in step $i$, we introduce exactly $n_i+s_i$ new edges, all of the same color, say $c$. If $d_i=0$, then this constitutes $r$ new edges of the same color, givings us $r-1=n_i+s_i-1$ new repetitions. If $d_i\neq 0$, then for some $k$, $\pi_k(\vec{u}_{i}\vec{v}_{i})\in E(H_k^{(i-1)})$, which implies that there already was an edge of color $c$ present in $H_k^{(i-1)}$, so all $n_i+s_i$ new edges are repetitions.

\subsection{Delayed Vertices}\label{sec::homo:delayed}

After performing the graph revealing algorithm, if there are many delayed vertices, then we do not expect to have as many repetitions as we desire. However, we also do not have as many vertices in $H\cup\pi(\vec{T})$ as we expected, and to capitalize on that, we wish to add more vertices to get more repetitions. In general, for every $r$ new vertices, we expect $r$ new edges, and thus $r-1$ new repetitions; if we have $D$ delayed vertices, we wish to get about $\frac{r-1}{r} D$ extra repetitions. The goal of this subsection is to make this more precise. We achieve these extra repetitions via an easy-to-find gadget in $\vec{G}$.

Let $H\subseteq G$. An \emph{$H$-reservoir with source $\vec{v}$} is a set $\vec{R}\subseteq V(\vec{G})$ along with a vertex $\vec{v}\in V(\vec{G})$ such that the following holds:
\begin{enumerate}
	\item $\pi(\vec{v})\subseteq H$,
	\item $\vec{u}\vec{v}\in E(\vec{G})$ for all $\vec{u}\in \vec{R}$,
	\item $H$ and $\pi(\vec{u})$ are disjoint for all $\vec{u} \in \vec{R}$.
\end{enumerate}

Note that if $\pi(\vec{v}) \subseteq F \subseteq H$ and $\vec{R}$ is an $H$-reservoir with source $\vec{v}$, then $\vec{R}$ is also an $F$-reservoir. We now show that $H$-reservoirs allow us to add repetitions we missed from delayed vertices.

\begin{lemma}\label{lem::reservoir}
	Let $H\subseteq G$, and let $\vec{R}$ be an $H$-reservoir. For any non-negative integer $D$ with $D\leq r|\vec{R}|$ there exists some graph $H^*\subseteq G$ that satisfies the following:
	\begin{itemize}
		\item $H\subseteq H^*\subseteq H\cup \pi(\vec{R})$,
		\item $|V(H^*)|=|V(H)|+D$, and
		\item $H^*$ has at least $\left\lfloor\frac{(r-1)}rD\right\rfloor$ more repetitions than $H$.
	\end{itemize}
\end{lemma}

\begin{proof}
	Let $\vec{v}$ be the source of the reservoir $\vec{R}$. Let $w,z$ be integers with $0\leq z<r$ such that $D=wr+z$, and note that
	\[
	\left\lfloor\frac{(r-1)}rD\right\rfloor=(r-1)w+z-\mathds{1}_{(z\neq 0)}.
	\]
	Choose $w$ vertices from $\vec{R}$, say $\vec{v}_1,\dots,\vec{v}_w$, and add $\pi(\vec{v} \vec{v}_i)$ to $H$ to form $H'$. Note that since $\vec{v}\vec{v_i}\in E(\vec{G})$ for all $1\leq i\leq w$, there are $r$ edges in $E(H')\setminus E(H)$ from vertices in $\pi(\vec{v})$ to vertices in $\pi(\vec{v_i})$, all of which are the same color, giving us $(r-1)w$ new repetitions. If $z=0$, then $H'$ satisfies the statement of the lemma, so we are done. If $z\neq 0$, then let $\vec{v}_{w+1}$ be any vertex in $\vec{R}\setminus\{\vec{v}_1,\dots,\vec{v}_w\}$. Add $\pi_k(\vec{v}_{w+1})$ to $H'$ for $1\leq k\leq z$ to form $H^*$. Note that this gives us a collection of $z$ more edges, all of the same color, yielding $z-1$ more repetitions, so $H^*$ satisfies the conditions of the lemma.
\end{proof}

\subsection{Constructing Graphs with the Right Number of Repetitions}\label{sec::homo:main}

\begin{theorem}\label{thm::homo with right reps}
	Let $H\subseteq G$ and $\vec{T}\subseteq \vec{G}$ be graphs. Fix some $H$-compatible ordering of $E(\vec{T})$, and let $\vec{R}$ be an $H\cup\pi(\vec{T})$-reservoir with $|\vec{R}|\geq \lceil D/r\rceil$. Let $m:=|E(\vec{T})|$. If
	\begin{equation}\label{eq::total savings}
	S + \sum_{i=1}^m \mathds{1}_{(d_i>0)}\left(\frac{r-d_i}{r-1}\right) \geq t,
	\end{equation}
	for some integer $t\geq 0$, then there exists a graph $H^*$ with $H\cup \pi(\vec{T})\subseteq H^*\subseteq H\cup \pi(\vec{T})\cup \pi(\vec{R})$, such that
	\[
	|V(H^*)|\leq |V(H)|+rm-t,
	\]
	and the number of repetitions in $H^*$ that are not in $H$ is at least 
	\[
	(r-1)m.
	\]
\end{theorem}

\begin{proof}
	Let $H'=H\cup \pi(\vec{T})$. By \eqref{eq::vertices from revealing},
	\[
	|V(H')|=|V(H)|+rm-S-D,
	\]
	and by \eqref{eq::repeats from revealing}, $H'$ has
	\[
	rm-D-d
	\]
	more repetitions than $H$ does. Let
	\[
	D':=S+D-t \geq D - \sum_{i=1}^m \mathds{1}_{(d_i>0)}\left(\frac{r-d_i}{r-1}\right).
	\]
	Apply Lemma~\ref{lem::reservoir} (with $H'$ as $H$, $\vec{R}$ as $\vec{R}$, and $D'$ as $D$) to find a graph $H^*$ with $H'\subseteq H^*\subseteq H'\cup\pi(\vec{R})$ such that
	\[
	|V(H^*)|=|V(H')|+D'=|V(H)|+rm-t,
	\]
	so $H^*$ has the correct number of vertices. In addition, $H^*$ has $\left\lfloor\frac{(r-1)}rD'\right\rfloor$ more repetitions than $H'$. This gives us that $H^*$ has at least
	\begin{align*}
	rm-D-d+\left\lfloor\frac{r-1}rD'\right\rfloor &\geq rm-D-d+\left\lfloor \frac{r-1}{r} \left(D - \sum_{i=1}^m \mathds{1}_{(d_i>0)}\left(\frac{r-d_i}{r-1}\right) \right) \right\rfloor\\
	&=(r-1)m+\left\lfloor m-\frac{D}r-d-\sum_{i=1}^m \mathds{1}_{(d_i>0)}\left(\frac{r-d_i}{r}\right)\right\rfloor\\
	&=(r-1)m+\left\lfloor m-\frac{D}r-d-(m-d)+\frac{D}r\right\rfloor\\
	&=(r-1)m
	\end{align*}
	repetitions that are not in $H$.
\end{proof}

In light of Theorem~\ref{thm::homo with right reps},  given graphs $H$ and $\vec{T}$ (with a fixed $H$-compatible ordering of $E(\vec{T})$), we define the \emph{total savings of $\vec{T}$ with respect to $H$ and the ordering $\sigma$}, $\mathrm{sav}(\vec{T},H,\sigma)$, to be
\[
\mathrm{sav}(\vec{T},H,\sigma):=S + \sum_{i=1}^m \mathds{1}_{(d_i>0)}\left(\frac{r-d_i}{r-1}\right).
\]

\subsection{Properties of the Graph Revealing Algorithm}\label{sec::homo:reveal}

In this subsection, we will provide some nice properties of the parameters we get from the graph revealing algorithm.

\subsubsection*{Order-Invariance}\label{sec::homo:reveal:order}

First we show that given graphs $H\subseteq G$ and $\vec{T}\subseteq\vec{G}$, many of the parameters given by the graph revealing algorithm are constant over all $H$-compatible orderings of $\vec{T}$.

\begin{observation}\label{obs::order invariance}
	Given graphs $H\subseteq G$ and $\vec{T}\subseteq \vec{G}$, for every $k\in [r]$ the quantities $N_k$, $S_k$ and $D_k$ are constant across all $H$-compatible orderings of $\vec{T}$. Consequently, the quantities $N$, $S$, and $D$ are also constant across all $H$-compatible orderings of $\vec{T}$.
\end{observation}

\begin{proof}
	Let $\sigma$ be an $H$-compatible ordering of $\vec{T}$. Since \[
	n_{i,k}(\vec{T},H,\sigma)+s_{i,k}(\vec{T},H,\sigma)+d_{i,k}(\vec{T},H,\sigma)=1,
	\]
	we have that $N_k+S_k+D_k=m$, so it will suffice to show that two of these parameters are constant (with respect to $\sigma$). First consider $N_k$, and note that
	\[
	N_k(\vec{T},H,\sigma)=|V(H\cup \pi_k(\vec{T}))|-|V(H)|,
	\]
	and the right side of the above equation is independent of $\sigma$. Now consider $D_k$. Given an edge $e\in E(G)$, let $\pi_k^{-1}(e)$ be the preimage of $e$ under $\pi_k$, or the set containing every edge of $\vec{T}$ that gets mapped to $e$. Then
	\[
	D_k(\vec{T},H,\sigma)=\sum_{e\in E(G)} \max\{0,|\pi^{-1}_k(e)|-\mathds{1}_{(e\notin E(H))}\}.
	\]
	Indeed, regardless of the ordering on $E(\vec{T})$, if $e\in E(H)$, then every edge of $\pi_k^{-1}(e)$ gives us one delayed vertex in coordinate $k$ on the step it is revealed. If $e\not\in E(H)$, then the first edge revealed in $\pi_k^{-1}(e)$ does not give us a delayed vertex in coordinate $k$, but all others do. Note again that the expression we derived for $D_k(\vec{T},H,\sigma)$ is independent of $\sigma$.
\end{proof}

It is worth noting that the parameter $d$ is not necessarily constant across all $H$-compatible orderings. In light of the preceding observation, in cases where we need to clarify $H$ and $\vec{T}$, we may write parameters such as $S_k(\vec{T},H)$, $D(\vec{T},H)$, and others without reference to $\sigma$. Since the parameters $n_{i,k}$, $s_{i,k}$, $d_{i,k}$, $n_i$, $s_i$, $d_i$ and $d$ all require the edge-ordering $\sigma$ to be well-defined though, we will specify $\sigma$ as appropriate.

\subsubsection*{Additivity}\label{sec::homo:reveal:add}

Let $\vec{T}_1,\vec{T}_2\subseteq \vec{G}$ be edge-disjoint graphs and $\sigma_1$ and $\sigma_2$ be $H$-compatible orderings of $\vec{T}_1$ and $\vec{T}_2$ respectively. Set $\vec{T}:=\vec{T}_1\cup\vec{T}_2$ and let $\sigma$ be the $H$-compatible ordering of $\vec{T}$ given by first ordering the edges of $\vec{T}_1$ in the order given by $\sigma_1$, then ordering the edges of $\vec{T}_2$ in the order given by $\sigma_2$. Then we have
\begin{equation}\label{eq::additivity}
	\mathrm{sav}(\vec{T},H,\sigma)=\mathrm{sav}(\vec{T}_1,H,\sigma_1)+\mathrm{sav}(\vec{T}_2,H\cup\pi(\vec{T}_1),\sigma_2).
\end{equation}

\subsubsection*{Monotonicity}\label{sec::homo:reveal:mono}

The parameters given by the graph revealing process also satisfy certain monotonicity properties with respect to subgraphs. More specifically, if $F_1\subseteq F_2\subseteq G$ and $\vec{T}\subseteq \vec{G}$ are graphs, and $\sigma$ is an $F_1$-compatible ordering of $\vec{T}$ (and $F_2$-compatible since $F_1\subseteq F_2$), then
\[
n_{i,k}(\vec{T},F_1,\sigma)\geq n_{i,k}(\vec{T},F_2,\sigma)
\]
and
\[
d_{i,k}(\vec{T},F_1,\sigma)\leq d_{i,k}(\vec{T},F_2,\sigma)
\]
for all $1\leq i\leq |E(\vec{T})|$ and $k\in [r]$. Indeed, if $n_{i,k}(\vec{T},F_2,\sigma)=1$, then $n_{i,k}(\vec{T},F_1,\sigma)=1$, since the second endpoint of $\pi_k(\vec{e}_i)$ cannot be in $F_1\cup\bigcup_{j=1}^{i-1}\pi(\vec{e}_j)$ if it is not in $F_2\cup\bigcup_{j=1}^{i-1}\pi(\vec{e}_j)$. Similarly, if $d_{i,k}(\vec{T},F_1,\sigma)=1$, then $d_{i,k}(\vec{T},F_2,\sigma)=1$, since  $\pi_k(\vec{e}_i)\subseteq F_1\cup\bigcup_{j=1}^{i-1}\pi(\vec{e}_j)$ implies $\pi_k(\vec{e}_i)\subseteq F_2\cup\bigcup_{j=1}^{i-1}\pi(\vec{e}_j)$ as well. Consequently, we have that
\begin{equation}\label{eq::Nmonotone}
N_k(\vec{T},F_1)\geq N_k(\vec{T},F_2)
\end{equation}
and
\begin{equation}\label{eq::Dmonotone}
D_k(\vec{T},F_1)\leq D_k(\vec{T},F_2).
\end{equation}

Unfortunately, $s_{i,k}$ does not satisfy such a monotonicity property. In general, steps in which we get a vertex savings with respect to $F_2$ may be new vertices with respect to $F_1$, and steps in which we get delayed vertices with respect to $F_2$ may be vertex savings with respect to $F_1$, so the number of vertex savings may increase or decrease when revealing $\vec{T}$ with respect to a subgraph.

\subsubsection*{Properties of Revealing Paths}\label{sec::homo:reveal:path}

Often the graph $\vec{T}$ we reveal with the graph revealing algorithm is a path (or a collection of paths). Given a path $\vec{P}\subseteq\vec{G}$ with endpoints $\vec{u}$ and $\vec{v}$, the \emph{canonical ordering of $\vec{P}$ from $\vec{u}$ to $\vec{v}$} is the ordering $\sigma$ of $E(\vec{P})=\{e_1,e_2,\dots,e_{\ell}\}$ such that $e_i$ appears before $e_{i+1}$ on $\vec{P}$ as we traverse $\vec{P}$ from $\vec{u}$ to $\vec{v}$. The following lemma will be very useful for finding vertex savings while revealing paths.

\begin{lemma}\label{lem::eventual savings}
	Let $F\subseteq G$. Let $\vec{P}=(\vec{v}_0,\vec{v}_1,\dots,\vec{v}_\ell)$ be a path in $\vec{G}$ such that $\pi(\vec{v}_0)\subseteq F$, and let $E(\vec{P})=\{\vec{e}_j\mid j\in [\ell]\}$ be canonically ordered from $\vec{v}_0$ to $\vec{v}_\ell$. If there exists a choice of $k\in [r]$ and $j,j'\in [\ell]$ with $j\leq j'$ such that
	\begin{enumerate}
		\item $\pi_k(\vec{e}_j)\not\in E\left(F\cup\bigcup_{i'=1}^{j-1}\pi_k(\vec{e}_{i'})\right)$, and
		\item $\pi_k(\vec{v}_{j'})\in V(F)$,
	\end{enumerate}
	then there is a vertex savings in coordinate $k$ with respect to $F$ as we reveal the path $\vec{P}$ at some step $j^*$ with $j\leq j^*\leq j'$.
\end{lemma}

\begin{proof}
	By (1.) above, step $j$ is either a new vertex or a vertex savings in coordinate $k$. If step $j$ is a vertex savings, then we are done. If not, and step $j$ is a new vertex, let $j^*>j$ be the first index such that 
	\[
	\pi_k(\vec{v}_{j^*})\in V\left(F\cup \bigcup_{i'=1}^{j^*-1}\pi_k(\vec{e}_{i'})\right).
	\] 
	Note that such an index exists and $j^*\leq j'$ since $\pi_k(\vec{v}_{j'})\in V(F)$. Then we claim step $j^*$ is a vertex savings in coordinate $k$. Indeed, step $j^*$ does not constitute a new vertex by the definition of $j^*$, and step $j^*$ cannot be a delayed vertex since step $j^*-1$ gives a new vertex with respect to $F$ (and thus by the $P_3$-preserving property of $\pi_k$, $\pi_k(\vec{e}_{j^*}) \not\in E(F\cup \bigcup_{i'=1}^{j^*-1}\pi_k(\vec{e}_{i'})$). Therefore, step $j^*$ gives a vertex savings as claimed, completing the proof.
\end{proof}

\section{Theta Graphs and Complete Bipartite Subdivisions}\label{sec::thetaKtt}

Recall that $\varTheta(a,b)$ consists of $b$ internally-disjoint paths of $a$ edges each with the same two endpoints, and $K_{a,b}^\ell$ is the graph obtained from $K_{a,b}$ by replacing each edge with a path with $\ell$ edges. In this section, we apply the techniques of Section~\ref{sec::homo} to show that $\varTheta(a,b)$ is $(r,\frac{a-1}a)$-nice for all $a > r \geq 2$, $b\geq 2$, and that $K_{3,b}^\ell$ is $(r,1-\frac{2}{3\ell})$-nice for $b,\ell \geq 3$, and $3 \leq r \leq 6$ with $r<\frac{3}2\ell$.

\subsection{Revealing Theta Graphs}\label{sec::theta}

For the following theorem, we will use the fact that for  $a,b\geq 2$, 
\begin{equation}\label{equation theta extremal number}
\ex(n, \varTheta(a,b)) = O(n^{1+1/a}),
\end{equation}
and that this result is tight when $b$ is sufficiently large compared to $a$~\cite{BT,FS}.

\begin{thm}\label{thm::theta}
Let $a>r\geq 2$. Then $\varTheta(a,b)$ is $(r,1-\frac{1}{a})$-nice for all $b\geq 2$ and $r$-nice when $b=2$ or $b$ is sufficiently large compared to $a$. Consequently, letting $\ell:=2+b(a-1)=|V(\varTheta(a,b))|$, we have
\[
f\left(n,r\ell,\binom{r\ell}{2}-(r-1)ab+1\right)=\Omega\left(n^{\frac{r}{r-1}\cdot\frac{a-1}a}\right).
\]
\end{thm}

\begin{proof} 
	Let $G=(V,E)$ be a complete graph $K_n$, let $C$ be a set of colors, and let $\chi:E\rightarrow C$ be an $(r\ell, \binom{r\ell}{2} - (r-1)ab+1)$-coloring of $G$, where $\ell = 2 + b(a-1)$. Since $\chi$ is an $(r\ell, \binom{r\ell}{2} - r\ell+3)$-coloring, by Proposition~\ref{prop::pruned exists} there exists a pruned $r$-th energy graph $\vec{G}=(\vec{V},\vec{E})$ of $G$.
	
	We address the cases $b=2$ and $b\geq 3$ separately. For the $b\geq 3$ case, we will only prove that $\Theta(a,b)$ is $(r,1-\frac{1}{a})$-nice. For $b$ sufficently large compared to $a$, this will give us that $\Theta(a,b)$ is $r$-nice since~\eqref{equation theta extremal number} is tight for $b$ large compared to $a$.
	
	For $b \geq 3$, we actually find $\varTheta(a,b')$ in the pruned energy graph, where $b'$ is much larger than $b$, and use this $\varTheta(a,b')$ to find a copy of $\varTheta(a,b)$ with a reservoir. For $b=2$, $\varTheta(a,b) = C_{2a}$, and the extremal number $\ex(n,C_{2a})$ is only known (up to constant factors) when $a=2$, $3$, or $5$. Thus for $b=2$, we must find the reservoir differently than for $b\geq 3$.
	
	\textbf{Case 1:} $b=2$. Let $\alpha>0$ be a constant such that $\mathrm{ex}(n,C_{2a})=O(n^{2-\alpha})$. Note that $\alpha\leq 1$. Let $C_{2a}\cup S_k$ be the graph on $2a+k$ vertices formed from the cycle $C_{2a}$ by adding $k$ degree $1$ vertices all adjacent to the same vertex in the cycle. We claim that for constant $k$,
	\[
	\mathrm{ex}(n,C_{2a}\cup S_k)\leq \mathrm{ex}(n,C_{2a})+2a(2a+k)n=\Theta(\mathrm{ex}(n,C_{2a})).
	\]
	Indeed, a graph on $n$ vertices with $\mathrm{ex}(n,C_{2a})+2a(2a+k)n$ edges must contain at least $2a(2a+k)n/2a=(2a+k)n$ edge-disjoint copies of $C_{2a}$, and so there must be a vertex $v$ in $2a+k$ copies. This implies that $v$ is in a copy of $C_{2a}$ and has degree at least $2a+k$, so even if $v$ has $2a-1$ neighbors on this cycle, this still leaves $k$ neighbors outside of the cycle, forming a copy of $C_{2a}\cup S_k$.
	
	Now, suppose for the sake of contradiction that $\vec{G}$ contains a copy $\vec{C}$ of $C_{2a}\cup S_{2a(r+1)}$. Let $\vec{v}$ denote the vertex of degree $2a(r+1)+2$ in $\vec{C}$, and $\vec{u}$ denote a neighbor of $\vec{v}$ in $\vec{C}$ of degree 2 in $\vec{C}$. Let $\vec{X}$ be the set of degree $1$ vertices in $\vec{C}$. We will let $H=\pi(\vec{u}\vec{v})$, and $\vec{T}$ be the $\vec{u}\text{--}\vec{v}$-path of length $2a-1$ in $\vec{C}$. Let $\vec{R}$ consist of those vertices of $\vec{X}$ whose coordinates are disjoint from $H \cup \pi(\vec{T})$, so that $\vec{R}$ is a $H\cup \pi(\vec{T})$-reservoir. We claim that $|\vec{R}|\geq 2a$. Note that there are at most $2ar$ vertices in $H\cup\pi(\vec{T})$. Since all vertices in $\vec{X}$ are distance at most two from each other, each vertex in $H\cup\pi(\vec{T})$ is a coordinate in at most one vertex in $\vec{X}$. This forbids at most $2ar$ vertices of $\vec{X}$ from being in $\vec{R}$, and hence $|\vec{R}| \geq 2a$. 
	
	Now, let us reveal $\vec{T}$, where $E(\vec{T})$ is canonically ordered from $\vec{u}$ to $\vec{v}$. We claim that $S(\vec{T},H)\geq r$. Indeed, by the $P_3$-preserving property, the first edge revealed constitutes a new vertex in all $r$ coordinates, and then by Lemma~\ref{lem::eventual savings}, since the endpoint $\pi_k(\vec{v})$ of $\pi_k(\vec{T})$ is in $V(H)$ for each $k\in [r]$, we get at least one savings in each coordinate, giving us $r$ savings altogether.
	
	Therefore, by Theorem~\ref{thm::homo with right reps} applied with $t=r$, there is a graph $H^*\subseteq G$ with
	\[
	|V(H^*)|\leq |V(H)|+r|E(\vec{T})|-r\leq 2r+r(2a-1)-r=2ar,
	\]
	and at least
	\[
	(r-1)|E(\vec{T})|=2a(r-1)
	\]
	color repetitions in $H^*$, contradicting the choice of coloring on $G$. Thus, $\vec{G}$ does not contain a copy of $C_{2a}\cup S_{2a(r+1)}$, so
	\[
	|E(\vec{G})|\leq \mathrm{ex}(n^r,C_{2a}\cup S_{2a(r+1)})=O\left(n^{r(2-\alpha)}\right).
	\]
	By Propositions~\ref{prop::color energy} and \ref{prop::pruned exists}, this implies $|C|=\Omega\left(n^{\frac{\alpha r}{r-1}}\right)$. This means that $C_{2a}$ is $(r,\alpha)$-nice for any $\alpha$, and thus $C_{2a}$ is $r$-nice.

	\textbf{Case 2:} $b\geq 3$. We wish to show that $\varTheta(a,b)$ is $(r,1-1/a)$-nice.
	Suppose for the sake of contradiction that $\vec{G}$ contains a copy of $\varTheta := \varTheta(a,2ra^2b)$. Let $\vec{v}_0,\vec{v}_a\in \vec{V}$ denote the vertices of degree greater than 2 in $\varTheta$. Label the paths $\vec{P}_{i} = (\vec{v}_0=\vec{v}_{i,0}, \vec{v}_{i,1}, \dots, \vec{v}_{i,a-1}, \vec{v}_{i,a}=\vec{v}_a)$ for $1 \leq i \leq 2ra^2b$. We may select a sequence of distinct paths $\vec{P}_{i_1}, \dots, \vec{P}_{i_{ab+b}}$ such that $\vec{v}_{i_j,1}$ has no coordinates in common with any vertices in $\{\vec{v}_0,\vec{v}_a\}\cup\bigcup_{k=1}^{j-1} V(\vec{P}_{i_k})$ for all $1 \leq j \leq ab+b$. This is possible because there are at most $(a-1)(j-1)+2 \leq (a-1)(ab+b)+2\leq 2a^2b$ vertices in this union, creating at most $2ra^2b$ possible coordinate conflicts with possible choices for $\vec{v}_{i_j,1}$. Since the vertices $\vec{v}_{i,1}$ for $1 \leq i \leq 2ra^2b$ are all distance $2$ from one another, our pruning guarantees that each possible coordinate conflict eliminates at most one choice of $\vec{v}_{i,1}$ for $\vec{v}_{i_j,1}$. Thus, there exists such a choice for $i_1, \dots, i_{ab+b}$. We reorder the paths so that these paths come first (that is, $i_j = j$ for all $1\leq j \leq ab+b$), and we discard the remaining paths.
	
	Now let $\vec{T}=\bigcup_{i=1}^b \vec{P}_{i}$, $H=\pi(\vec{v}_0)\cup \pi(\vec{v}_a)$, and $\vec{R}=\{\vec{v}_{i,1}\mid b+1\leq i\leq b+ab\}$. Let $\sigma$ be an ordering of $E(\vec{T})$ such that the edges of $\vec{P}_{i}$ appear before the edges of $\vec{P}_{j}$ for all $1\leq i<j\leq b$, and within each path, the edges are given the canonical ordering from endpoints $\vec{v}_0$ to $\vec{v}_a$. By the ordering placed on the paths in $\varTheta$, the vertices in $\vec{R}$ do not have any coordinates in common with each other or any vertices in $\vec{T}$, so $\vec{R}$ is a $H\cup\pi(\vec{T})$-reservoir with source $\vec{v}_0$.
	
	We claim that $S(\vec{T},H)\geq rb$. In fact, each path $\vec{P}_{i}$ gives us $r$ vertex savings when revealed with respect to $H\cup\bigcup_{j=1}^{i-1}\pi(\vec{P}_{j})$. Indeed, since $\pi_k(\vec{v}_a)\in H$ and since the edge $\pi_k(\vec{v}_0\vec{v}_{i,1})\not\in H\cup\bigcup_{j=1}^{i-1}\pi(\vec{P}_{j})$ by the ordering placed on the paths, Lemma~\ref{lem::eventual savings} implies that there is a vertex savings in coordinate $k$ as we reveal $\vec{P}_{i}$ for each $k\in [r]$. Thus
	\[
	\mathrm{sav}(\vec{T},H,\sigma)\geq rb.
	\]
	Therefore, by Theorem~\ref{thm::homo with right reps} applied with $t=rb$, there is a graph $H^*\subseteq G$ with
	\[
	|V(H^*)| \leq |V(H)|+r|E(\vec{T})|-rb = 2r+rab-rb=r\ell
	\]
	and at least
	\[
	(r-1)|E(\vec{T})|=(r-1)ab
	\]
	repetitions. This contradicts the choice of coloring on $G$, so $\vec{G}$ does not contain a copy of $\varTheta(a,2ra^2b)$. Thus,
	\[
	|E(\vec{G})|\leq \mathrm{ex}(n^r,\varTheta(a,2ra^2b))=O\left(n^{r+\frac{r}a}\right)
	\]
	and hence by Propositions~\ref{prop::color energy} and \ref{prop::pruned exists}, $|C|=\Omega\left(n^{\frac{r}{r-1}\cdot\frac{a-1}{a}}\right)$.
\end{proof}

\subsection{Revealing Subdivided Complete Bipartite Graphs}\label{sec::subKtt}

To prove the following theorem, we will use the bound  $\ex(n,K_{a,b}^{\ell})=O(n^{1+\frac{a-1}{a\ell}})$ given by Janzer~\cite{J2}. 

\begin{theorem}\label{thm::subKtt}
Let $b\geq 3$, $\ell \geq 2$, and $2 \leq r \leq 6$ with $r<\frac{3}{2} \ell$. Then $K_{3,b}^\ell$ is $(r,1-\frac{2}{3\ell})$-nice. In other words, letting $s:=|V(K_{3,b}^{\ell})|=3+b+3(\ell-1)b$ and noting that $|E(K_{3,b}^{\ell})|=3b\ell$, we have
\[
f\left(n,rs,\binom{rs}2-3(r-1)b\ell+1\right)=\Omega\left(n^{\frac{r}{r-1}\left(1-\frac{2}{3\ell}\right)}\right).
\]
\end{theorem}

\begin{proof}
Let $G=(V,E)$ be a complete graph $K_n$, let $C$ be a set of colors, and let $\chi:E\rightarrow C$ be an $(rs, \binom{rs}{2} - 3(r-1)b\ell+1)$-coloring of $G$, where $s=|V(K_{3,b}^{\ell})|= 3+b+3(\ell-1)b$. Since $\chi$ is an $(rs, \binom{rs}{2} - rs+3)$-coloring, by Proposition~\ref{prop::pruned exists}, there exists a pruned $r$-th energy graph $\vec{G}=(\vec{V},\vec{E})$ of $G$.

Assume to the contrary that there is a copy $\vec{K}$ of $K_{3,30rb\ell^2}^\ell$ in $\vec{G}$. Let $\{\vec{a}_1,\vec{a}_2,\vec{a}_3\}$ be the set of vertices of degree $30rb\ell^2$ in $\vec{K}$, $\{\vec{b}_1,\dots,\vec{b}_{30rb\ell^2}\}$ be the set of vertices of degree $3$ in $\vec{K}$,  and $\vec{x}_{j,i}$ and $\vec{y}_{j,i}$ denote the first and last vertices (that are not $\vec{a}_j$ or $\vec{b}_i$) along the $\vec{a}_j\text{--}\vec{b}_i$ geodesic in $\vec{K}$. We will let $\vec{P}_{j,i}$ denote this geodesic. 
Let $\vec{Q}_i$ denote the subgraph of $\vec{K}$ induced by $\bigcup_{j=1}^3 V(\vec{P}_{j,i})$, which we call a \emph{page}.

We may select a sequence of distinct pages $\vec{Q}_{i_1}, \dots, \vec{Q}_{i_{3b(1+\ell)}}$ such that $\vec{x}_{j,i_p}$ has no coordinates in common with any vertex of $\vec{Q}_{i_{p'}}$ for any $1 \leq p' < p \leq 3b(1+\ell)$ and $1\leq j\leq 3$. This is possible because there are at most 
\[
\sum_{p'=1}^{p-1}|V(\vec{Q}_{i_{p'}})|\leq (3\ell+1)(p-1) \leq (3\ell+1)\cdot 3b(1+\ell)
\]
vertices whose coordinates we must exclude from choices of $\vec{x}_{j,i_p}$. Since the vertices $\vec{x}_{j,i}$ with the same $j$ are all distance two from each other, they must have distinct coordinates, so at most $(3r)(3\ell+1)(3b(1+\ell))< 30rb\ell^2$ choices for $\vec{Q}_{i_p}$ are excluded, and we can choose this sequence $i_1, \dots, i_{3b(1+\ell)}$. We reorder the pages so that these pages come first, that is, $i_p = p$ for all $1 \leq p \leq 3b(1+\ell)$, and we discard the remaining pages.

As in the proof of Theorem~\ref{thm::theta}, our goal is to apply Theorem~\ref{thm::homo with right reps}. It would be convenient to apply Theorem~\ref{thm::homo with right reps} to a graph $\vec{T}\subseteq \vec{G}$ where $\vec{T}\cong K_{3,b}^\ell$, however the total savings 
\[
S+\sum_{i=1}^m \mathds{1}_{(d_i>0)}\left(\frac{r-d_i}{r-1}\right)
\]
from \eqref{eq::total savings} may not be large enough for our purposes in this case. Instead, we will choose a graph $\vec{T}$ more carefully. In general, we will look at the pages $\vec{Q}_i$ one at a time, and if adding a page to $\vec{T}$ will increase the total savings by $2r$ (the amount we would expect from revealing this page if we were revealing color-isomorphic copies of $K_{3,b}^{\ell}$ in all $r$ coordinates), we will do so. If this does not happen, we will instead try to find a path $\vec{P}_{j,i}$, such that adding this path increases \eqref{eq::total savings} by $\frac{2r}3$. 

In order to help keep track of the total savings, we will build the graph $\vec{T}$ in $b$ `chapters,' where each chapter is three consecutive pages of $\vec{K}$.  Depending on how a chapter interacts in $G$ with the previously-revealed chapters, we determine which parts of those three pages to add to $\vec{T}$. Thus, we will use the pages $\vec{Q}_i$ for $1\leq i\leq 3b$ in our $b$ chapters. The remaining $3b\ell$ pages will yield a $\pi(\bigcup_{i=1}^{3b}\vec{Q}_i)$-reservoir of size $3b\ell$. Indeed, note that by the ordering we placed on the pages, $\vec{R}:=\{\vec{x}_{1,i}\mid 3b+1\leq i\leq 3b\ell\}$ is a $\pi(\bigcup_{i=1}^{3b}\vec{Q}_i)$-reservoir with source $\vec{a}_1$.

Let $\vec{T}_0$ be the empty graph on vertex set $V(\vec{T}_0)=\{\vec{a}_1,\vec{a}_2,\vec{a}_3\}$, and let $H^{(0)}:=\pi(\vec{T}_0)$. We call an ordering on any subset of $E(\vec{T})$ \emph{consistent} if it satisfies the following properties:
\begin{itemize}
    \item the edges on any path $\vec{P}_{j,i}$ are ordered canonically from $\vec{a}_j$ to $\vec{b}_i$,
    \item the edges of a path $\vec{P}_{j,i}$ appear before the edges of $\vec{P}_{j',i'}$ whenever $i<i'$ or $i=i'$ and $j<j'$.
\end{itemize}
Note that the first property implies that, for subsets of $E(\vec{T})$ which are unions of paths $\vec{P}_{j,i}$, consistent orderings are $H^{(0)}$-compatible.

We will assume throughout the rest of the proof that all subgraphs are revealed according to a consistent ordering, and will therefore drop any reference to orderings in our notation.

We will recursively define graphs $\vec{T}_1,\vec{T}_2,\dots,\vec{T}_b\subseteq \vec{K}$ such that the following hold:
\begin{itemize}
	\item $\mathrm{sav}(\vec{T}_i,H^{(0)}\cup\bigcup_{j=1}^{i-1}\pi(\vec{T}_j))\geq 2r$,
	\item $\vec{T}_i=\vec{P}_{j_1,i_1}\cup\vec{P}_{j_2,i_2}\cup\vec{P}_{j_3,i_3}$, where $3i-2\leq i_1,i_2,i_3\leq 3i$, and $j_1,j_2,j_3\in [3]$.
\end{itemize}

More specifically, $\vec{T}_i$ will either be equal to one of the pages $\vec{Q}_{3i-2},\vec{Q}_{3i-1},\vec{Q}_{3i}$, or $\vec{T}_i$ will contain the first path from each of these three pages. Let us assume that $\vec{T}_{i'}$ has been chosen to satisfy the above requirements for all $i'<i$, and let us define for each $i'\leq i$,
\[
H^{(i')}:=H^{(0)}\cup\bigcup_{j=1}^{i'-1}\pi(\vec{T}_{j}).
\] 
If for some $i'$ with $3i-2\leq i'\leq 3i$, we have
\begin{equation}\label{inequality ordinary page}
\mathrm{sav}\left(\vec{Q}_{i'},H^{(i-1)}\right)\geq 2r,
\end{equation}
then set $\vec{T}_i:=\vec{Q}_{i'}$. (If there is more than one choice of $i'$, choose one arbitrarily.) Otherwise, set $\vec{T}_i:=\bigcup_{i'=3i-2}^{3i}P_{1,i'}$.

Note that when $\ell=2$, we have $\mathrm{sav}(\vec{Q}_{i'},H^{(i-1)})\geq 2r$ for all $3i-2\leq i'\leq 3i$. To see why, note that in this case, $\vec{x}_{j,i'}=\vec{y}_{j,i'}$ for each $1\leq j\leq 3$, so not only do we have $\vec{x}_{j,i'}\notin H^{(i-1)}$, but also $\vec{x}_{j,i'}\neq\vec{x}_{j',i'}$ for $j\neq j'$. Therefore, $\pi_k(\vec{a}_2\vec{x}_{2,i'})$ is not in $H^{(i-1)}\cup\pi(\vec{P}_{1,i'})$ for any $k$, and Lemma~\ref{lem::eventual savings} guarantees that $S(\vec{P}_{2,i},H^{(i-1)})\geq r$. Repeating this argument with $\vec{P}_{3,i'}$ shows that $S(\vec{Q}_{i'},H^{(i-1)})\geq 2r$. 

When $\ell\geq 3$, the following technical lemma ensures that our choice of $\vec{T}_i$ gives at least $2r$ total savings.

\begin{lemma}\label{lem::extraordinary}
    Let $3i-2 \leq i' \leq 3i$. If
    \[
	\mathrm{sav}\left(\vec{Q}_{i'},H^{(i-1)}\right) < 2r,
	\]
    then,
    \[
	\mathrm{sav}\left(\vec{P}_{1,i'},H^{(i-1)}\cup\bigcup_{j'=3i-2}^{i'-1}\pi(\vec{P}_{1,j'})\right) \geq 2r/3.
	\]
\end{lemma}

We delay the proof of Lemma~\ref{lem::extraordinary} until we finish this proof of Theorem~\ref{thm::subKtt}. Lemma~\ref{lem::extraordinary} implies that for all $1\leq i\leq b$, there is a choice of $\vec{T}_i$ such that
\[
\mathrm{sav}(\vec{T}_i,H^{(i-1)})\geq 2r.
\]
Let $\vec{T}:=\bigcup_{i=1}^b \vec{T}_i$. Then, by repeated application of \eqref{eq::additivity}, we have that
\[
\mathrm{sav}(\vec{T},H^{(0)})=\sum_{i=1}^b\mathrm{sav}(\vec{T}_i,H^{(i-1)})\geq 2rb.
\]
Furthermore, recall that $\vec{R}$ is a $\pi(\bigcup_{i=1}^{3b}\vec{Q}_i)$-reservoir, and since $H^{(0)}\cup\pi(\vec{T})\subseteq \pi(\bigcup_{i=1}^{3b}\vec{Q}_i)$, $\vec{R}$ is an $H^{(0)}\cup\pi(\vec{T})$-reservoir of size $3b\ell\geq \lceil D(\vec{T},H^{(0)})/r\rceil$. Since $|E(\vec{T}_i)|=3\ell$, we have $|E(\vec{T})|=3b\ell$. Then Theorem~\ref{thm::homo with right reps} (applied with $H^{(0)}$ as $H$ and $2rb$ as $t$), gives us a graph $H^*\subseteq G$ with
\begin{align*}
|V(H^*)|\leq |V(H^{(0)})|+r|E(\vec{T})|-t\leq 3r+3rb\ell-2rb=r|V(K_{3,b}^\ell)|
\end{align*}
and at least
\[
(r-1)|E(\vec{T})|=(r-1)3b\ell
\]
repetitions, contradicting the choice of coloring on $G$. Thus, $\vec{G}$ is $K_{3,30rb\ell^2}^\ell$-free, so 
\[
|E(\vec{G})|\leq \mathrm{ex}\left(n^r,K_{3,30rb\ell^2}^\ell\right)=O\left(n^{r+\frac{2r}{3\ell}}\right). 
\]
Therefore, by Propositions~\ref{prop::color energy} and \ref{prop::pruned exists}, $|C|=\Omega\left(n^{\frac{r}{r-1}\left(1-\frac{2}{3\ell}\right)}\right)$.
\end{proof}

\begin{proof}[Proof of Lemma~\ref{lem::extraordinary}]

Fix $i'$ with $3i-2 \leq i' \leq 3i$, and assume that
\[
\mathrm{sav}\left(\vec{Q}_{i'},H^{(i-1)}\right) < 2r.
\]
Our goal is to show that the path $\vec{P}_{1,i'}$ gives at least $2r/3$ total savings. To prove this, we need to establish a couple of smaller technical lemmas.

\begin{lemma}\label{lem::purepaths}
    If $N_k(\vec{P}_{1,i'}, H^{(i-1)}) = \ell$, then
    $S_k(\vec{P}_{2,i'} \cup \vec{P}_{3,i'}, H^{(i-1)} \cup \pi(\vec{P}_{1,i'})) \geq 2$.
\end{lemma}

\begin{proof}
    To simplify notation, let $H = H^{(i-1)}$, $\vec{b} = \vec{b}_{i'}$, $\vec{P}_j = \vec{P}_{j,i'}$, $\vec{x}_j = \vec{x}_{j,i'}$, and $\vec{y}_j = \vec{y}_{j,i'}$ for $j \in [3]$. 
    By Observation~\ref{obs::order invariance}, we are done unless $S_k(\vec{P}_2, H \cup \pi(\vec{P}_1)) \leq 1$. Since $N_k(\vec{P}_{1}, H) = \ell$, $\pi_k(\vec{b})$ is a new vertex, so this vertex has degree $1$ in $H\cup \pi_k(\vec{P}_1)$. Then, since $\pi_k(\vec{y}_2) \neq \pi_k(\vec{y}_1)$, the edge $\pi_k(\vec{y}_2 \vec{b})$ does not appear in $H \cup\pi_k(\vec{P}_1)$. Thus by Lemma~\ref{lem::eventual savings}, $S_k(\vec{P}_2, H \cup \pi(\vec{P}_1)) = 1$.
    
    Now we claim that $\pi_k(\vec{b})$ has degree at most two in $H\cup \pi_k(\vec{P}_1)\cup \pi_k(\vec{P}_2)$. Indeed, $\pi_k(\vec{b})$ has only $\pi_k(\vec{y}_1)$ as a neighbor in $H\cup \pi_k(\vec{P}_1)$, so if we assume towards a contradiction that $\pi_k(\vec{b})$ has degree at least $3$ in $H\cup \pi_k(\vec{P}_1)\cup \pi_k(\vec{P}_2)$, then $\pi_k(\vec{P}_2)$ must  visit $\pi_k(\vec{b})$ at least twice. 
    
    Since $\pi_k(\vec{P}_1)$ is length $\ell$, this implies that $\pi_k(\vec{P}_2)$ deviates from $\pi_k(\vec{P}_1)$ prior to visiting $\pi_k(\vec{b})$, encountering at least one edge not in $H\cup\pi_k(\vec{P}_1)$. Thus by Lemma~\ref{lem::eventual savings}, $\vec{P}_2$ gets a savings prior to visiting $\pi_k(\vec{b})$ for the first time. Let $\vec{P}_2'$ be the path containing all the edges of $\vec{P}_2$ coming before and including the edge that first revisits $\pi_k(\vec{b})$. Then $\pi_k(\vec{b})$ has degree at most two in $H\cup \pi_k(\vec{P}_1)\cup\pi_k(\vec{P}_2')$, and the only edges it can be adjacent to are the last edges of $\pi_k(\vec{P}_1)$ and $\pi_k(\vec{P}_2')$. Thus,  when  $\pi_k(\vec{P}_2)$ encounters a third edge adjacent to $\pi_k(\vec{b})$ after getting one savings,  Lemma~\ref{lem::eventual savings} implies that $\vec{P}_2$ will get a second savings in coordinate $k$, a contradiction of our assumption that $S_k(\vec{P}_2, H \cup \pi(\vec{P}_1)) \leq 1$.
    
    Finally, we claim $S_k(\vec{P}_3, H \cup \pi(\vec{P}_1 \cup \vec{P}_2)) \geq 1$, which will give us our second savings and finish the proof. Indeed, since $\pi_k(\vec{y}_1)$, $\pi_k(\vec{y}_2)$ and $\pi_k(\vec{y}_3)$ are all distinct, $\pi_k(\vec{b})$ must have degree at least $3$ in $H\cup\pi_k(\vec{P}_1\cup\vec{P}_2\cup\vec{P}_3)$, so $\pi_k(\vec{P}_3)$ must encounter a new edge adjacent to $\pi_k(\vec{b})$, and thus by Lemma~\ref{lem::eventual savings}, $\vec{P}_3$ gets at least one savings in coordinate $k$.
\end{proof}

\begin{lemma}\label{lem::path w one s,d}
    If $S_k(\vec{P}_{1,i'},H^{(i-1)})=1$ and $D_k(\vec{P}_{1,i'},H^{(i-1)})\leq 1$, then $S_k(\vec{Q}_{i'},H^{(i-1)})\geq 2$.
\end{lemma}

\begin{proof}
    Let $H = H^{(i-1)}$, $\vec{P}_j = \vec{P}_{j,i'}$, $\vec{b}=\vec{b}_{i'}$, $\vec{x}_j = \vec{x}_{j,i'}$, and $\vec{y}_j = \vec{y}_{j,i'}$ for $j \in [3]$. Let $P$ denote the initial segment of $\pi_k(\vec{P}_1)$ that terminates at the vertex in which $\vec{P}_1$ gets its vertex savings, and note that since $D_k(\vec{P}_1,H)\leq 1$, $P$ must contain at least $\ell-1$ edges. We are done unless $S_k(\vec{P}_2\cup\vec{P}_3, H \cup \pi(\vec{P}_1)) = 0$. Since $\pi_k(\vec{P}_2)$ and $\pi_k(\vec{P}_3)$ terminate in $H \cup \pi(\vec{P}_1)$, Lemma~\ref{lem::eventual savings} implies that when revealing $\vec{P}_2$ and $\vec{P}_3$, we only have delayed vertices.
    
    \textbf{Case 1:} $\pi_k(\vec{b}) \in V(H)$. Since $\pi_k(\vec{x}_2)$ and $\pi_k(\vec{x}_3)$ are not in $H$,  the $P_3$-preserving property implies that an initial segment of $\pi_k(\vec{P}_2)$ and $\pi_k(\vec{P}_3)$ share the same edges as $P$, either traversing the edges of $P$ in order or in reverse. Note that since all three of the walks $\pi_k(\vec{P}_j)$ cover the edges of $P$ in this way, some pair of these walks must traverse $P$ in the same order, say $\pi_k(\vec{P}_{j^*_1})$ and $\pi_k(\vec{P}_{j^*_2})$ with $j^*_1<j^*_2$. Since $P$ has at least $\ell-1$ edges, there is at most one edge in $\pi_k(P_{j^*_1})$ or $\pi_k(P_{j^*_2})$ that is not in $P$: the terminal edge. As both $\pi_k(P_{j^*_1})$ and $\pi_k(P_{j^*_2})$ terminate at $\pi_k(\vec{b})$ and are identical except possibly at the final edge, they are actually the same walk. This implies $\pi_k(\vec{y}_{j^*_2}) = \pi_k(\vec{y}_{j^*_1})$, a contradiction. 
    
    \textbf{Case 2:} $\pi_k(\vec{b}) \not\in V(H)$. If the degree of $\pi_k(\vec{b})$ in $H \cup \pi(\vec{P}_1)$ is less than $3$, either $\pi_k(\vec{y}_2 \vec{b})$ or $\pi_k(\vec{y}_3\vec{b})$ was not revealed when we revealed $\vec{P}_1$, so via Lemma~\ref{lem::eventual savings}, there will be a second vertex savings when either $\vec{P}_2$ or $\vec{P}_3$ are revealed. The only remaining situation to consider is when $\pi_k(\vec{b})$ has degree at least $3$ in $H\cup\pi(\vec{P}_1)$. This implies that $\pi_k(\vec{P}_1)$ never returns to a vertex in $H$ and visits $\pi_k(\vec{b})$ twice, getting a vertex savings on the very last step, implying that $P=\pi_k(\vec{P}_1)$, and that there are exactly $\ell$ edges in $P$, and the only edge of $P$ that contains a vertex in $H$ is the initial edge. Then since $\pi_k(\vec{x}_2)$ and $\pi_k(\vec{x}_3)$ are not in $H$ and neither $\vec{P}_2$ nor $\vec{P}_3$ have any vertex savings, $\pi_k(\vec{P}_2)$ and $\pi_k(\vec{P}_3)$ have the same initial edge as $\pi_k(\vec{P}_1)$, and then must contain only edges in $\pi_k(\vec{P}_1)$ until possibly returning to $H$. However, it is impossible for $\pi_k(\vec{P}_2)$ or $\pi_k(\vec{P}_3)$ to return to $H$, since these are $P_3$-preserving walks of length $\ell$, and $P$ has length $\ell$. There are only two walks of length $\ell$ on the edges of $\pi_k(\vec{P}_1)$ ending in $\pi_k(\vec{b})$, one of which is $\pi_k(\vec{P}_1)$ itself, so two of the walks $\pi_k(\vec{P}_j)$ are identical. However, the walks $\pi_k(\vec{y}_j)$ are distinct, so this is a contradiction.
\end{proof}

    With these two lemmas in hand, let us return to the proof of Lemma~\ref{lem::extraordinary}. 

	Let $H^*:=H^{(i-1)}\cup\bigcup_{j'=3i-2}^{i'-1}\pi(\vec{P}_{1,j'})$. We wish to calculate the total savings for the page $\vec{Q}_{i'}$, so we will consider the $r$ different coordinates of this page separately.
	
	We will call the coordinate $k\in [r]$ \emph{good} if $S_k(\vec{Q}_{i'},H^{(i-1)})\geq 2$, and \emph{bad} otherwise. Note that if all $k$ coordinates are good, $\mathrm{sav}(\vec{Q}_{i'},H^{(i-1)})\geq 2r$, so we are done unless we have at least one bad coordinate. Fix $k\in [r]$ and assume coordinate $k$ is bad. By the contrapositive of Lemma~\ref{lem::purepaths}, $N_k(\vec{P}_{1,i'},H^{(i-1)})<\ell$. Recall that $\pi_k(\vec{x}_{1,i'}) \not\in H^{(i-1)}$, so we have $N_k(\vec{P}_{1,i'},H^{(i-1)})>0$. Since new vertices can only be followed by new vertices or savings, then $S_k(\vec{P}_{1,i'},H^{(i-1)}) = 1$. Furthermore, by Lemma~\ref{lem::path w one s,d}, $D_k(\vec{P}_{1,i'},H^{(i-1)})\geq 2$, so the vertex savings does not happen on the last or second-to-last step of revealing $\vec{P}_{1,i'}$. 
	
	We claim that after this vertex savings happens, every other step of revealing $\vec{Q}_{i'}$ will consist of a delayed vertex in coordinate $k$. Indeed, since there cannot be a second vertex savings in coordinate $k$, if there is a step that is not a delayed vertex, it must be a new vertex. First consider the possibility that the step in which we reveal the edge $\vec{y}_{1,i'}\vec{b}_{i'}$ is a new vertex. In this case, the edge $\pi_k(\vec{y}_{2,i'}\vec{b}_{i'})$ cannot be in $H^{(i-1)}\cup \pi_k(\vec{P}_{1,i'})$. Thus, by Lemma \ref{lem::eventual savings}, as we reveal $\vec{P}_{2,i'}$, we encounter a new edge, while $\pi_k(\vec{b}_{i'})$ already has been revealed, so we get a second savings in coordinate $k$, contradicting that coordinate $k$ is bad. 
	
	Thus, the final step of revealing $\vec{P}_{1,i'}$ is a delayed vertex, and so every step between the vertex savings and this final step is also a delayed vertex. Furthermore, once $\vec{P}_{1,i'}$ is revealed, if $E(\pi_k(\vec{P}_{2,i'}\cup\vec{P}_{3,i'}))\setminus E(H^{(i-1)}\cup\pi(\vec{P}_{1,i'}))\neq \emptyset$, Lemma~\ref{lem::eventual savings} gives us a second savings in coordinate $k$, since $\pi_k(\vec{b}_{i'})$ has been revealed. Thus, every step after the vertex savings in coordinate $k$ gives us a delayed vertex in coordinate $k$.
	
	Now we can calculate $\mathrm{sav}(\vec{Q}_{i'},H^{(i-1)})$. Let there be $b^*$ bad coordinates and $r^*$ coordinates $k$ such that $N_k(\vec{P}_{1,i'},H^{(i-1)})=\ell$. Note that we can assume $b^*\geq 1$. We claim that we are done unless $r^*\geq \lfloor r/3\rfloor+1$. Indeed, if not, note that since $N_k$ is monotone with respect to subgraphs (see \eqref{eq::Nmonotone}), there are at most $\lfloor r/3\rfloor$ coordinates $k$ such that $N_k(\vec{P}_{1,i'},H^*)=\ell$, and thus, there are at least $r-\lfloor r/3\rfloor\geq 2r/3$ coordinates $k$ with $S_k(\vec{P}_{1,i'},H^*)\geq 1$. Therefore, we have 
	\[
	S(\vec{P}_{1,i'},H^*)\geq 2r/3,
	\]
	as desired.
	
	Now, since each bad coordinate has a vertex savings, and all other coordinates have at least two vertex savings,
	\[
	S(\vec{Q}_{i'},H^{(i-1)})\geq b^*+2(r-b^*)=2r-b^*.
	\]
	Furthermore, we established that the last two steps of revealing $\vec{P}_{1,i'}$ in each bad coordinate give us a delayed vertex, while in the $r^*$ coordinates with no vertex savings in the first path, these last two steps are not delayed vertices, so these two steps contribute at least
	\[
	\frac{2r^*}{r-1}
	\]
	to the total savings. Now, let us consider
	\[
	\sum_{i^*}\mathds{1}_{(d_{i^*}>0)}\left(\frac{r-d_{i^*}}{r-1}\right),
	\]
	where the sum is taken over all steps $i^*$ that correspond to revealing $\vec{P}_{2,i'}$ and $\vec{P}_{3,i'}$. By Lemma~\ref{lem::purepaths}, when revealing $\vec{P}_{2,i'}$ and $\vec{P}_{3,i'}$, each of the $r^*$ coordinates with no savings in the first path give us two vertex savings, which occurs after $\vec{P}_{1,i'}$ has already been revealed, and consequently these two steps do not constitute delayed vertices. Thus,
	\begin{align*}
	\sum_{i^*}\mathds{1}_{(d_{i^*}>0)}\left(\frac{r-d_{i^*}}{r-1}\right)=\frac{\sum_{i^*}(r-d_{i^*})}{r-1}\geq \frac{2r^*}{r-1},
	\end{align*}
	where the first equality follows since we have at least one bad coordinate, and in that coordinate, there is a delayed vertex in every step $i^*$ considered in the sum. 
	We now calculate the total savings. We have that
	\begin{equation}\label{inequality final savings bound}
	\mathrm{sav}(\vec{Q}_{i'},H^{(i-1)})\geq 2r-b^*+\frac{4r^*}{r-1}.
	\end{equation}
	Using the fact that $b^*\leq r-r^*$ and that $r^*\geq \lfloor r/3\rfloor +1$, the bound in \eqref{inequality final savings bound} is always at least $2r$ when $r\in \{3,4\}$, so we obtain a contradiction in these cases.
	
	Now, let us assume that $r\in\{5,6\}$. Note that our earlier analysis gives us that $b^*\leq 3$, however in this case, we claim that $b^*\leq 2$. To see this, let us assume to the contrary that $b^*=3$. Now, let us consider $\mathrm{sav}(\vec{P}_{1,i'},H^*)$. 
	Due to the monotonicity of delayed vertices with respect to subgraphs, see \eqref{eq::Dmonotone}, we know that when revealing $\vec{P}_{1,i'}$ with respect to $H^*$, the final two steps still give us delayed vertices in each of the three bad coordinates. This implies that
	\[
	S(\vec{P}_{1,i'},H^*)\geq 3,
	\]
	since the first step of revealing $\vec{P}_{1,i'}$ must constitute a new vertex in every coordinate, and so we must have at least one vertex savings in each bad coordinate before we can encounter the delayed vertices in the last two steps of revealing the path (Lemma~\ref{lem::eventual savings}). Furthermore, the total vertex savings must be exactly $3$ here, since $2r/3\leq 4$ for $r\in \{5,6\}$, and we are done unless $\vec{P}_{1,i'}$ does not have $2r/3$ total savings. As we have already observed, this implies that the $r^*$ coordinates that are not bad give us new vertices at every step of $\vec{P}_{1,i'}$, thus, based on the delayed vertices in the bad coordinates, we can see that
	\[
	\mathrm{sav}(\vec{P}_{1,i'},H^*)\geq 3+2\frac{3}{r-1}\geq 4\geq 2r/3.
	\]
	
	Otherwise, $b^*\leq 2$ in this case, and we still have that $r^*\geq \lfloor r/3\rfloor+1$, so we see that in this case, the right side of \eqref{inequality final savings bound} is at least $2r$ for $r\in\{5,6\}$, completing the proof.
\end{proof}

\section{Future Work}\label{sec::conc}

We believe that Theorem~\ref{thm::subKtt} can be generalized further, although more case analysis and new ideas are needed.
\begin{conjecture}
	For positive integers $a,b,\ell,r$ with $a\leq b$ and $\frac{r}{r-1}(1-\frac{a-1}{a\ell})\geq 1$, the graph $K_{a,b}^{\ell}$ is $r$-nice.
\end{conjecture}
If a graph $F$ is $r$-nice, then upper bounds for $\ex(n,F)$ yield lower bounds on $f(n,p,q)$. As mentioned in Section~\ref{sec::intro}, in the contrapositive, this means upper bounds on $f(n,p,q)$ yield lower bounds on $\ex(n,F)$. It could be interesting to further explore this connection. For instance, Theorem~\ref{thm::theta} shows that $C_{2k}$ is $r$-nice, so obtaining new upper bounds on $f(n,2rk, \binom{2rk}{2} - 2(r-1)k + 1)$ would give lower bounds on $\ex(n,C_{2k})$, a well-known hard problem. The smallest interesting example of this implication is for $r=2$ and $k=4$. Proving $f(n,16,\binom{16}{2}-7) = O(n^{3/2})$ would imply that $\ex(n,C_8) = \Theta(n^{5/4})$. Currently the best upper bound for $f(n,16,\binom{16}{2}-7)$ is $O(n^{7/4})$, given by the local lemma bound in \eqref{eq::lll}.

Theorem~\ref{thm::PS} from \cite{PS} exploited another connection between $f(n,p,q)$ and Tur\'an numbers (not using the color energy graph). This connection appeared implicitly in \cite{PS}, so we make the connection explicit here.

\begin{lemma}\label{lem::turan}
Let $1<\gamma<2$ be fixed, and let $F$ be a bipartite graph with bipartition $V(F)=A\cup B$ which contains at least two edges. Suppose that every subgraph of $K_{n,n^\gamma}$ with $(\frac{n^2}{4})/(|A|+|E(F)|)$ edges contains a copy of $F$ with $A$ on the side of size $n$ and $B$ on the side of size $n^\gamma$. Then for $n$ sufficiently large,
\[
f\left(n,|A|+|E(F)|, \binom{|A|+|E(F)|}{2} - (|E(F)|-|B|)+1\right) \geq n^{\gamma} .
\]
\end{lemma}

\begin{proof}
Let $G = (V,E)$ be a complete graph $K_n$, let $C$ be a set of colors, and let $\chi: E \to C$ be an $(|A|+|E(F)|, \binom{|A|+|E(F)|}{2} - (|E(F)|-|B|)+1)$-coloring of $G$. Note that when $p:=|A|+|E(F)|$, this coloring is above the linear threshold, $\binom{p}2-p+3$, so the color degrees in $G$ are bounded by $p-1=|A|+|E(F)|-1$. Suppose for the sake of contradiction that $|C| \leq n^{\gamma}$. First, form $G'$ by deleting from $G$ all edges whose color appears on fewer than $n^{\frac{2-\gamma}2}$ edges. In doing so, we delete at most $|C|n^{\frac{2-\gamma}2}\leq n^{\frac{2+\gamma}2}=o(n^2)$ edges total, so $G'$ still has $n^2/4$ edges.

We form the \emph{color incidence graph}, which is a bipartite graph with parts $V$ and $C$, and an edge between $v \in V$ and $c \in C$ if there is an edge of $G'$ incident to $v$ with color $c$. Since each vertex of $G'$ is incident to at most $|A|+|E(F)|-1$ edges of a given color, there are at least $(\frac{n^2}{4})/(|A|+|E(F)|)$ edges in the color incidence graph. By our hypothesis, this implies that there is a copy of $F$ in the color incidence graph with $A \subseteq V$ and $B \subseteq C$.

This copy of $F$ gives us a clique with many color repetitions in $G'$. In $G'$, for each vertex $v\in A$, there is a star, $S_v$, centered at $v$ with $d_F(v)$ edges, each of which are colored with a color from $B$. Let $S:=\bigcup_{v\in A} S_v$. If each of the stars $S_v$ are pairwise edge-disjoint for all $v\in A$, then $S$ contains at most
\[
\sum_{v\in A}|V(S_v)|=|A|+\sum_{v\in A}d_F(v)=|A|+|E(F)|
\]
vertices, and exactly
\[
\sum_{v\in A}|E(S_v)|-|B|=|E(F)|-|B|
\]
color repetitions, which contradicts our choice of coloring. Therefore, these stars are not pairwise edge-disjoint, and we let $a = \sum_{v\in A} |E(S_v)| - |E(S)|$. Then $S$ has at most
\[
\sum_{v\in A}|V(S_v)|-2a=|A|+\sum_{v\in A}d_F(v)-2a=|A|+|E(F)|-2a
\]
vertices and exactly
\[
|E(S)|-|B|=|E(F)|-a-|B|
\]
color repetitions. Since every color class in $G'$ has at least $n^{\frac{2-\gamma}{2}}>\binom{|V(S)|}2+a$ edges, we can choose a set of $a$ edges, disjoint from $E(S)$ that are colored with a color in $B$. These $a$ edges span at most $2a$ vertices, so adding them to $S$ gives us a graph on at most $|A|+|E(F)|$ vertices with $|E(F)|-|B|$ color repetitions, again a contradiction. Thus, $|C|>n^\gamma$.
\end{proof}

There is one famous example of an upper bound for asymmetric bipartite Tur\'an numbers: $F = K_{s,t}$ \cite{KST}. In this case, Lemma~\ref{lem::turan} recovers Theorem~\ref{thm::PS}. While other upper bounds for asymmetric bipartite Tur\'an numbers are known (for example, for theta graphs), we have found no further applications of Lemma~\ref{lem::turan} which give improvements on existing bounds. It could be fruitful to further explore this connection between $f(n,p,q)$ and asymmetric bipartite Tur\'an numbers.

\section*{Acknowledgements}

We thank the anonymous referees for their many useful comments and suggestions. This research was performed while the third author was at the University of Illinois Urbana-Champaign.

\bibliographystyle{abbrv}
\bibliography{arXiv_version}

\end{document}